\documentclass[review,onefignum,onetabnum]{siamart190516}

\usepackage{lipsum}
\usepackage{amsfonts}
\usepackage{graphicx}
\usepackage{epstopdf}
\usepackage{algorithmic}

\usepackage[left=3cm,right=3cm,top=2cm,bottom=2cm]{geometry}
\usepackage{tikz} 
\usetikzlibrary{decorations.markings,calc,fit,intersections}  
\usepackage{pgfplots} 
\usepackage{mathtools} 
\usepackage{amssymb}  
	
\ifpdf
  \DeclareGraphicsExtensions{.eps,.pdf,.png,.jpg}
\else
  \DeclareGraphicsExtensions{.eps}
\fi

\newsiamremark{example}{Example}  
\newsiamremark{remark}{Remark}
\newsiamremark{hypothesis}{Hypothesis}
\crefname{hypothesis}{Hypothesis}{Hypotheses}
\newsiamthm{claim}{Claim}
\nolinenumbers
\headers{Weak Measure-Valued Solutions of a Nonlinear HCL}{X. Gong, and M. Kawski}

\title{Weak Measure-Valued Solutions of a Nonlinear Hyperbolic
       Conservation Law
\thanks{Submitted to the editors 12/25/2019.
}}

\author{Xiaoqian Gong\thanks{Center for Computational and Integrative Biology,
Rutgers University, Camden, NJ
  (\email{xg143@scarletmail.rutgers.edu}).}
\and Matthias Kawski\thanks{School of Mathematical and Statistical Sciences,
Arizona State University, Tempe, AZ,
  (\email{kawski@asu.edu}  ).}}
 
\usepackage{amsopn}

\ifpdf
\hypersetup{
  pdftitle={\title{Weak Measure-Valued Solutions to a Nonlinear Hyperbolic
       Conservation Law
}
},
  pdfauthor={X. Gong, and M. Kawski}
}
\fi

\begin{document}

\maketitle

\begin{abstract}
We revisit a well-established model for highly re-entrant
semi-conductor manufacturing systems, and analyze it in the
setting of states, in- and outfluxes being Borel measures.
This is motivated by the lack of optimal solutions in the
$L^1$-setting for transitions from a smaller to a larger
equilibrium with zero backlog.
Key innovations involve dealing with discontinuous velocities
in the presence of point masses, and a finite domain with in-
and outfluxes.
Taking a Lagrangian point of view, we establish existence and
uniqueness of solutions, and formulate a notion of weak solution.
We prove continuity of the flow with
respect to time (and almost also with respect to the initial
state). Due to generally discontinuous velocities, these
delicate regularity results hold only with respect to carefully
crafted semi-norms that are modifications of the flat norm. 
Generally. the solution is not continuous with respect to
any norm on the space of measures.
\end{abstract}



\section{Introduction}
\label{sect1}
In the context of data and solutions that are Borel measures,
this article reinterprets, analyzes, and proves well-posedness of
the scalar nonlinear controlled hyperbolic conservation law
(for fixed $T>0$)
\begin{subequations}
\label{npeqn1}
\begin{align}
\label{npeqn1a}
0&= \partial_t \rho(t,x) + \partial_x (\alpha(W(t))\rho(t,x))
\;\mbox{ for }\; (t,x)\in [0,T] \times [0,1],
\\
\label{npeqn1b}
W(t)&=\int_0^1 \rho(t,x) \,dx \;\mbox{ for }\; t\in [0,T],
\\
\label{npeqn1c}
\rho_0(x) &= \rho (0,x) \;\mbox{ for }\; x\in [0,1],
\\
\label{npeqn1d}
u(t)&= \rho (t,0) \alpha(W(t)) \;\mbox{ for }\; t\in [0,T],\;\mbox{ and }
\\
\label{npeqn1e}
y(t)&= \rho (t,1) \alpha(W(t)) \;\mbox{ for }\; t\in [0,T].
\end{align}
\end{subequations}
This system was introduced in~\cite{AM2006} to model
highly re-entrant semiconductor manufacturing systems.
Here $t$ represents time,
$x$ the degree of completion of the product,
and $\rho$ the density of the {\em work in progress}.
The velocity $\alpha$ is a decreasing
nonnegative function of the total load $W$ of the factory.
This reflects the highly re-entrant character of semi-conductor
manufacturing systems, in which each {\em part} revisits the same
machines again and again as additional layers are added to each part.
The original work~\cite{AM2006} validated the model
using numerical simulations, comparing with discrete-event
systems, and with real factory data.
Given an initial density $\rho_0$, the first main problem
is to find a control influx $u$ such that
the outflux $y$ tracks a given reference {\em forecasted demand signal} $y_d$.
The corresponding optimal control problem of finding
a control influx $u$ that minimizes an
error signal such as $J(u)=\|y-y_d\|$ was studied in
\cite{herty2010}, which,
in the context of $L^2$-data $\rho_0,\;u,\;y_d$
and an $L^2$-control objective derived adjoint equations,
and numerically computed approximations of optimal controls.

The article~\cite{CKW2010} proved well-posedness for the
Cauchy problem~\eqref{npeqn1a}-\eqref{npeqn1d} (disregarding the outflux)
in the context of $L^1$-data (thus implying well-posedness for $L^2$-data),
analyzed the regularity of solution curves
$\rho\colon [0,T]\mapsto L^1([0,1])$, and established existence
of optimal controls for the original $L^2$-problem.
Well-posedness of a related multi-dimensional uncontrolled
problem with unbounded spatial dimension and was also demonstrated
in~\cite{Herty2011}, proving local existence of a weak entropy solutions
and examining differentiability with respect to initial data.

Both, more meaningful from the business perspective,
and mathematically more challenging is the optimal
control problem with an $L^1$-objective.
From a business perspective, an important problem is the transfer
between equilibria with zero backlog (at the terminal time)
and minimizing the $L^1$-norm of the backlog,
that is, the difference between the desired cumulative outflux
$Y_d(t)=\int_0^ty_d(\tau)\,d\tau$ and the actual cumulative outflux
$Y(t)=\int_0^ty(\tau)\,d\tau$.
While the complete solution remains elusive, partial analytical
results and numerical studies indicate that for the zero-backlog
problem no optimal controls exists in $L^1([0,T])$
(for the transfer from a smaller to a larger equilibrium): Minimizing
sequences converge to impulsive controls~\cite{XGthesis,hyp18}.
Thus it is natural to recast the problem in the setting of
controls and states being measures.

In recent years the analysis of similar hyperbolic conservation
laws in the setting of measures has seen substantial attention
and progress.
Here we briefly mention a few, and point the interested readers
to related references in these articles.
Motivated by earlier work on interactions of densities and point masses
in the context of prey and predators~\cite{predprey},
the article~\cite{shortpaper} established the well-posedness
of similar nonlinear hyperbolic conservation laws~\eqref{npeqn1}
with non-local velocity in the setting of measure-valued data.
Other articles are motivated as models for pedestrian or vehicular
traffic. The Wasserstein metric is a popular tool for models
that use probability measures, usually on unbounded domains.
In order to allow for sources, and nonconstant total mass
a generalized Wasserstein metric was introduced
and studied in ~\cite{numerical, GWD}.
Closely related are the Kantorovich-Rubinstein norm
and the dual bounded Lipschitz-norm or flat norm, see~\cite{Thieme2018}
for a careful study of continuity of semiflows on the space
of Borel measures endowed with the flat norm.
The article \cite{thethesis} introduces an innovative concept
of sticky boundaries to deal with flux boundary conditions, while
the article \cite{Evers16} considers a flow that stops at the boundaries.
Other very recent closely related articles~\cite{keimer2017,keimer2018}
consider system with the velocity being a {\em weighted} functional
of the work in progress.

The problem addressed in this article has several distinctive features
that significantly set it apart from the recent literature.
Foremost, due to generally the influx being different from the outflux,
the total mass is not constant.
Consequently, most tools available for probability measures such
as the Wasserstein metric do not apply here.
Even more importantly, as a characteristic feature of the highly
re-entrant semiconductor manufacturing system~\cite{AM2006},
the velocity depends on the total load as in~\eqref{npeqn1b},
whereas in most popular traffic models it is governed
by local interactions which are modeled by convolutions,
which naturally smoothen the velocity.
However, in the factory model impulsive influxes and outfluxes cause the total load,
and hence the velocity, to be discontinuous as functions of time.
Consequently, weak solutions of~\eqref{npeqn1}
are no longer meaningful in the usual distributional sense
(formal integration by parts, discontinuous velocities multiplying
Dirac distributions).

Staying close to the original manufacturing system modeled by~\eqref{npeqn1}
this article reinterprets the hyperbolic conservation law~\eqref{npeqn1a}-\eqref{npeqn1c}.
The key is to temporarily
abandon the Eulerian point of view, and instead focus on the
Lagrangian point of view, which tracks locations of {\em parts}
(or {\em particles}).
Subsection~\ref{sect21} motivates the reinterpretation by coupling
the hyperbolic conservation law~\eqref{npeqn1a} for $L^1$-densities
with a sequence of ordinary differential equations that track the
location of {\em concentrations} resulting from impulses in the
control influx, or already present in the initial data.
The remainder of section~\ref{sect2} formalizes the problem statement
in the context of Borel measures, disambiguates terminology, fixes
notation, and assembles some technical tools needed later.
Section~\ref{sect3} establishes the existence of unique solutions
for the Cauchy problem from the Lagrangian point of view for
the system with data that are Borel measures.
A key step is to obtain a uniform lower bound for the lengths
of time-intervals on which we can construct contraction mappings.
This involves an innovation for dealing with large point masses
leaving the system at a-priori unknown times.
The construction takes advantage of the characteristic curves
being bi-Lipschitz, a key feature of the model.
Section~\ref{sect4} defines a notion of weak solution
for the system~\eqref{npeqn1} with data that are Borel measures,
demonstrates that the Lagrangian solutions are weak solutions,
and that the weak solutions are unique.
Due to the lack of continuity, the space of test functions has
to be significantly modified from the standard spaces.
Finally, regularity properties of the solutions are proven.
Given the noncompact domain, these are expressed using weighted
versions of the flat norm, defined in terms of semi-norms that
account for impulses entering and leaving the system. Note that
the solution is not continuous with respect to the flat norm.

Staying close to the features of the original manufacturing system
modeled by~\eqref{npeqn1}, provides both advantages that suggest
carefully tailored approaches, but also leads to technical
complications that prevent application of standard tools.
In particular, we chose the underlying spaces
for our measures to be the noncompact intervals $[0,1)$ and $(0,T]$.
This is essential for obtaining the desired contractions needed for
fixed point arguments. Another simple and strong argument for this
choice is the common practical choice of the CONWIP dispatch policy
({\em ``constant work in progress''}) for factories that are performing
well: use the output feedback law $u=y$, influx equals outflux.
With impulsive outfluxes, if using $[0,1]$ instead, this would lead
to awkward total loads  which are constant except for possibly
countably many jump discontinuities.

\section{Motivation, disambiguation of terminology, and technical preparation}
\label{sect2}

\subsection{Motivating the Lagrangian point of view:
A PDE coupled with two sequences of ODEs}
\label{sect21}

In \cite{XGthesis} it was shown that for a transfer from a smaller to a larger
equilibrium state no $L^1$ influx $u$ (with small support) minimizes the
$L^1$-objective function $J(u)=\int_0^T\int_0^t|(y(s)-y_d(s)|\,ds\,dt$
(the tracking error between the actual and the desired accumulated outfluxes,
for arbitrary but fixed $T>0$).
Minimizing sequences {\em tend to impulsive controls}.
This suggests that one consider influxes that are densities together with
multiple, possibly, countably many of such point masses.
This is similar in spirit to ~\cite{predprey}, which combined
densities of prey together with isolated    predators (point masses).
One possible model for this might incorporate a combination of
a hyperbolic conservation law for an $L^1$-density as in~\eqref{npeqn1a}
together with a sequence of ODEs,
all of which are coupled by the total mass and velocity.

Suppose the initial state $\rho_0$ consists of a
function $\rho_{0,L^1}\in L^1([0,1))$ and a sequence
of point masses $m_i$ located at pairwise distinct
$x_i \in [0,1)$, $i = 1, 2, \cdots $
with $\sum\limits _{i=1}^\infty m_i <\infty$.
Similarly, the influx $u$ consists of a function $u_{L^1}\in L^1((0,T])$
and a sequence of point masses $M_j$
entering the system at pairwise distinct  times $t_j \in (0,T]$,
$j = 1, 2, \cdots $
with $\sum\limits_{j=1}^\infty M_j <\infty$.
Furthermore, let $\xi_i : [0, T] \to [0, +\infty)$
and $\eta_j : [t_j, T] \to [0, +\infty)$ trace the location
of the masses $m_i$ and $M_j$, respectively.
This suggests the following coupled model
\begin{subequations}
\label{p1eqn5}
\begin{align}
\label{p1eqn5a}
0&= \partial_t \rho(t,x) + \partial_x (\alpha(W(t))\rho(t,x))
\;\mbox{ for }\; (t,x)\in [0,T] \times [0,1),
\\
\label{p1eqn5a1}
\xi_i'(t)&=\alpha(W(t)) \;\mbox{ for almost all }\; t\in [0,T]
           \mbox{ and }\; i\in \mathbb{Z}^+,
\\
\label{p1eqn5a2}
\eta_j'(t)&=\alpha(W(t))\;\mbox{ for  almost all}\; t\in [t_j,T]
           \mbox{ and }\; j\in \mathbb{Z}^+,
\\
\label{p1eqn5b}
W(t)&=\int_0^1 \rho(t,x) \,dx \; +\!\!\!\!\sum_{\{i\colon  \xi_i(t)<1\}} \!\!\!\!\!\!\!\!m_i
\;\;+\!\!\!\sum_{\{j\colon t\geq t_j \mbox{ and }\eta_j(t)<1\}} \!\!\!\!\!\!\!\!\! M_j \mbox{ for }\; t\in [0,T],
\\
\label{p1eqn5c}
\rho_{0,L^1}(x) &= \rho (0,x) \;\mbox{ for }\; x\in [0,1),
\\
\label{p1eqn5c1}
\xi_i(0) &= x_i \;\mbox{ for }\;i\in \mathbb{Z}^+,
\\
\label{p1eqn5c2}
\eta_j(t_j) &= 0 \;\mbox{ for }\;j\in \mathbb{Z}^+,
\;\mbox{ and}
\\
\label{p1eqn5d}
u_{L^1}(t)&= \rho (t,0) \alpha(W(t)) \;\mbox{ for }\; t\in [0,T].
\end{align}
\end{subequations}

Note that for every fixed $t\in [0,T]$ the velocity $\alpha(W(t))$
is constant with respect to the location $x\in [0,1)$.
Thus there really is only a single ordinary differential equation.
All $\xi_i$ and $\eta_j$ are translates of each other
(up to restrictions of the domains).

This model, which combines densities with point-masses, reflects
natural features of the optimal control problem as dictated by the
original manufacturing system and business objective.
Mathematically, one may prefer combine the densities
and point masses and informally write
$u_{L^1}+\sum_j M_j\delta_{t_j}$ and
$\rho_{0,L^1}+\sum_i m_i\delta_{x_i}$
(with $\delta_s$ denoting the Dirac measure centered at $s$).
This model ~eqref{p1eqn5} shall serve as a mental reference when posing 
the model~\eqref{eq3} fin the context of Borel measures. The critical step is the interpretation
of the influx, relating measures in time to measures in space,
as products discontinuous velocities and Dirac measures
are not well-defined.

\subsection{Notation, and disambiguation of terminology}
\label{sect22}
Before continuing to frame the motivational problem~\eqref{p1eqn5}
in terms of measures, we take a moment to fix notation,
disambiguate terminology, and review some basic facts
about measurable functions and their compositions with
measurable, and with continuous functions that will be used
in the sequel.

For an interval $I\subseteq \mathbb{R}$ denote by $\mathcal{M}(I)$ and
$\mathcal{M}^{+}(I)$ the spaces of signed, respectively,
nonnegative finite Borel measures on $I$.
From now on we assume that the initial datum
$\rho_0 \in \mathcal{M}^{+}([0,1))$
and the control influx $\mu \in \mathcal{M}^{+}((0,T])$
 -- we write $\mu$ instead of $u$ -- are
finite nonnegative Borel measures.
Solutions of the problem will be curves
$\rho\colon [0,T] \mapsto \rho_t \in \mathcal{M}^{+}([0,1))$
that satisfy a differential equation in a sense to be defined.
Their regularity will be addressed in subsection~\ref{sect43}.
We may interchangeably use either $\rho_t$ and $\rho(t)$
depending on what makes for easier reading.

Throughout we assume that the velocity  is a strictly decreasing function 
$\alpha \colon [0,\infty)\mapsto (0,1]$ (of the load~$W$) normalized by $\alpha(0)=1$,
that is positive (hence bounded away from zero on compact sets),
and is Lipschitz continuous with Lipschitz constant $L$:
for all $W_1, W_2 \geq 0$, $|\alpha(W_1)-\alpha(W_2)| \leq L |W_1-W_2|$.
Whereas the article~\cite{AM2006}
used $\alpha(W) = \max\{ 0,1-\frac{W}{W_0}\}$,
a much more common choice in subsequent work was
$\alpha (W)=\frac{1}{1+W}$, which we use in some examples.
However, the general results presented here only use the above
stated properties of $\alpha$.

Recall the common conflicts of using the term {\em measurable function}
in  different settings: In abstract measure spaces, a function
$f\colon (A,\Sigma_1)\mapsto (B,\Sigma_2)$ is measurable if for
every (measurable) set $Y\in \Sigma_2$ the preimage $f^{-1}(Y)$
is measurable, i.e., $f^{-1}(Y)\in \Sigma_1$. For example,
a function is Borel measurable in this sense if preimages
of Borel sets are again Borel sets. In contrast,
it is common to say that a function
$f\colon \mathbb{R} \mapsto \mathbb{R}$ is (Lebesgue) measurable
if the preimage of every ray $(a,\infty)$,
and thus of every Borel subset of $\mathbb{R}$ is (Lebesgue) measurable.
Since the Lebesgue measure on the real line is regular,
Borel measures on real line are also regular
(both inner and outer regular: measurable sets are approximated
by compact sets from the inside, and by open sets from the outside).
Thus the refined Lebesgue decomposition theorem
also applies to the Borel measures in $\mathcal{M}^{+}$.

In the first, abstract setting compositions of measurable functions are measurable.
For instance, compositions of Borel measurable functions are again
Borel measurable. Furthermore, all monotone functions are Borel measurable.
In particular, in the sequel flow-lines $(t;r,x_0)\mapsto X(t; r, x_0)$
(defined in \eqref{defflow}) will play a key-role, as they map Borel sets
to Borel sets. Due to their Lipschitzness and their strict monotonicity
in each variable, they are always Borel measurable.

In the second setting of Lebesgue measure on the real line,
not even compositions of measurable and continuous functions
need be measurable, as is demonstrated by the well-known
example constructed from Cantor's devil's staircase function.
For a measure space $(X, \mathcal{M}, \mu)$,
a non-negative measurable function $f$ on $X$, $f$
is said to be integrable over $X$ (with respect to $\mu$)
provided $\int_{X}f\,d\mu < \infty$.
Furthermore, suppose $f\colon \mathbb{R}\mapsto \mathbb{R}$ is continuous
and $g\colon [a,b]\mapsto\mathbb{R}$ is Lebesgue integrable.
If there exist constants $c$ and $d$ such that for all $x\in [a,b]$,
$\;\left|f(x)\right| < c+d\left|x\right|$,
then $f\circ g$ is Lebesgue integrable over $[a,b]$.
In particular, in the proof to the uniqueness of the weak solution
of \eqref{eq3} (theorem \eqref{thmunique}),
the composition of $\hat{\xi}$ is well-defined since $\hat{W}$
is integrable and $\alpha$ is continuous and bounded.

\subsection{Posing the problem for data and states that are Borel measures}
\label{sect23}

It is mathematically more satisfactory than the complicated set-up in
section~\ref{sect21} to combine the $L^1$-densities
and point masses into measures and consider a single hyperbolic
conservation law like~\eqref{npeqn1} for data and states that are
Borel measures.
As briefly noted in the introduction, the choices of the half-open
intervals reflect the desire to have a constant total mass, to not
double count point masses entering and exiting at the same time,
when using the most simple output feedback $y=u$ that equates the
influx with the outflux. Moreover, these choices are instrumental
for technical arguments to obtain contractions.

{\bf Blanket assumption:}
We assume throughout that all the singular continuous parts of all
initial data and influxes are zero, i.e., they are sums of only an
absolutely continuous measure (with respect to Lebesgue measure)
and a pure point measure (a countable sum of positive multiples
of Dirac measures).

This assumption is motivated by the original industrial optimal
control problem where singular continuous measures seem to not
make much sense, and the desire to avoid unnecessary technical
complications in the sequel. This is well in line with much of
the recent literature, e.g.~\cite{numerical,GWD}.
It will be easy to see (lemma \ref{LemNosc}) that due to
characteristic curves being strictly monotone and bi-Lipschitz,
this implies that for every time~$t\in [0,T]$ the state
$\rho_t=\rho(t,\,\cdot\,)$ has the same property.

We restate the problem~\eqref{npeqn1} in the context of Borel measures,
with a purely formal first equation~\eqref{eq3a}
(because it is customary to first have a differential equation
before calling a curve a solution.)
For every fixed $\rho_0\in \mathcal{M}^{+}([0,1))$ and
$\mu\in \mathcal{M}^{+}((0,T])$, consider the
problem of finding a curve
$\rho\colon [0,T]\mapsto \mathcal{M}^{+}([0,1))$
and a map $X \colon \{(t;r,x)\colon 0\leq r\leq t\leq T,\,0\leq x\}
\mapsto [0,\infty)$ that satisfy in a sense to be established:
\begin{subequations}
\label{eq3}
\begin{align}
\label{eq3a}
0&= \partial_t \rho(t)
 +\partial_x (\alpha(W(t))\rho(t)) \;\mbox{ for a.e. }\; t\in [0,T],
\\
\label{eq3b}
W(t)&= \rho(t) ([0,1)) \;\mbox{ for all }\; t\in [0,T],
\\
\label{eq3c}
\rho_0& =\rho (0),
\\
\nonumber
\rho(t)(E) &= \mu(\{ r\in (0,t] \colon X(t;r,0)\in E \})+
              \rho_0(\{ x\in [0,1) \colon X(t;0,x)\in E \}),
\\
\label{eq3d}
&\;\;\; \;\mbox{ for all }\; t\in [0,T], \mbox{ and every  Borel set } E\subset [0,1),
\\
\label{eq3e}
\frac{d}{dt}X(t;r,x)&=\alpha(W(t)) \;\mbox{ for almost every }\; 0\leq r \leq t \leq T,
\;\mbox{ and all } x\in [0,1), \;\mbox{ and}
\\
\label{eq3f}
X(r;r,x)&=x \;\mbox{ for all}\; r\in [0,T]\;\mbox{ and all } x\in [0,1).
\end{align}
\end{subequations}
Equation~\eqref{eq3d} takes the role of the boundary
condition~\eqref{npeqn1d}, relating the {\em influx}
$\mu\in \mathcal{M}^{+}((0,T])$ to the state
$\rho_t\in \mathcal{M}^{+}([0,1))$.
It captures the sense of conservation of mass even better
than the partial differential equation \eqref{eq3a}.
However, this problem statement comes at the cost of presupposing
part of the form of the Lagrangian solution defined in the next section.

Informally, to connect system~\eqref{eq3} for Borel measures
to system~\eqref{npeqn1} for integrable functions,
the boundary condition \eqref{eq3d} might be interpreted
(in terms of the refined Lebesgue decomposition
$\mu = \mu_{ac}+\mu_{pp}$ and $\rho_t = \rho_{t,ac}+\rho_{t,pp}$)
where $\tilde{\rho}_{t,ac} = \left[ \frac{d\rho_{t,ac}}{d\lambda}\right]$ and
$u_{L^1} = \left[ \frac{d\mu_{ac}}{d\lambda}\right] $ are the Radon-Nikodym derivatives
of $\rho_{t,ac}$ and $\mu_{ac}$ with respect to Lebesgue measure $\lambda$.
For all Lebesgue measurable sets, and thus also for all
Borel sets $E \subset (0,T]$ and $F\subset[0,1)$, these satisfy
\begin{equation}
\mu_{ac}(E)=\int_{E} u_{L^1} d\lambda
\quad  \text{ and } \quad
\rho_{t,ac}(F) = \int_{F} \tilde{\rho}_{t,ac} d\lambda.
\end{equation}
From this point of view, the pure point part simply copies
from the time to the space direction, whereas the velocity
multiplies the $L^1$-functions associated to the absolutely
continuous parts - which is commensurate with $\rho_t$ being
the pushforward of $\mu$ by the semiflow as defined
in the next section.
\begin{equation}
\mu_{pp}(\{t\}) = \rho_{t,pp}(\{0\})
\quad  \text{ and } \quad
\tilde{u}_{L^1}(t)  = \tilde{\rho}_{t,ac}(0)\alpha(W(t)).
\end{equation}
Of course, as $L^1$ functions, the Radon-Nikodym derivatives
only have values at Lebesgue points. It could well be that,
e.g., $\tilde{\rho}_{t,ac}(0)$ is not defined for any $t$ at all,
i.e., if $x=0$ is not a Lebesgue point of $\tilde{\rho}_{t,ac}$
for  any $t\in [0,T]$.
Thus we consider this latter only an informal discussion, to motivate
the precise statement of notions of solutions in the forthcoming sections,
in particular, the choice how to deal with influx and outflux.

\section{Characteristic Curves and Lagrangian solutions}
\label{sect3}

\subsection{Overview of the argument, presentation, and organization}
\label{sect30}

Traditionally contraction mapping arguments for the existence of
unique solutions of differential equation decide on the
time interval (on which a specific map is a contraction)
at the end of the argument.
Due to the delicacy of our argument, and several expert readers
misunderstanding an earlier presentation of our argument in the
traditional form, we decided to reorganize our proof.
The main objective is to prevent any suspicion of a circular
argument, and to clearly delineate the strategy of constructing
a solution of a PDE in a space of measures from a solution of
a scalar nonlinear ODE with Borel measures as parameters.
The crux is to deal with {\em large} point masses exiting
from the system (outflux) at a-priori unknown times which
coincide with discontinuities of the velocity.

\subsection{Existence of Unique Short Time Solutions}
\label{secxi}
In this subsection, we first prove local existence of unique solutions
of a related scalar ordinary differential equation using a contraction
mapping argument.
For any fixed $\rho_0\in {\cal M}^+([0,1))$
and $\mu\in {\cal M}^+((0,T])$ consider the
Cauchy problem.
\begin{equation}
\dot{\xi}(t) = \alpha (\mu((0,t]) + \rho_0([0,1-\xi(t))))
\;\mbox{ for }\; 0\leq t\leq 1, \;\mbox{ together with }\;\xi(0)=0.
\end{equation}

While still depending on the influx $\mu$,
the key difference is that this is a scalar ordinary
differential equation with the single fixed initial measure
$\rho_0\in {\cal M}^+([0,1))$ as a parameter,
rather than a hyperbolic conservation law for measures.
(For small times the measure $\mu$ drops out in the contraction mapping argument.)
However, this formulation facilitates handling the discontinuities
of the velocity $\alpha (\mu((0,t]) + \rho_0([0,1-\xi(t))))$.

A key insight is that the contraction mapping argument,
can accommodate even an infinite number of discontinuities
of the time-varying vector field $v$, which coincide
with the times when point masses exit from the system.
However, one needs to {\em restart} the argument at times
when the velocity $v$ has {\em large} discontinuities,
caused by large point masses entering of exiting the system.
Critical is that the number of {\em large} point masses is
finite. A main technical issue is that it is not a-priori
known when these will leave the system -- i.e., the domain
of the set of curves to which one wants to apply a contraction
mapping principle is unknown until after it has been used
(on a system with modified parameters).
\\

\textbf{The Hypotheses for the Contraction Mappings}\\

Let $T>0$, and $\alpha\colon [0,\infty)\mapsto (0,1]$ be a strictly
decreasing  Lipschitz continuous function with $\alpha(0)=1$ and
fix a specific Lipschitz constant $L>0$ for $\alpha$.
Let $\rho_0 \in \mathcal{M}^{+}([0,1))$ and
$\mu \in \mathcal{M}^{+}((0,T])$ be arbitrary but fixed measures.
Fix
\begin{equation}
\label{vmindef}
v_{\rm min} = \alpha (1+\rho_0([0,1))+\mu([0,T]))>0.
\end{equation}
Denote the refined Lebesgue decomposition of
the initial condition $\rho_0$ and the influx $\mu$ by
\begin{equation}
\rho_0 = \rho_{0,ac}+\rho_{0,pp} \;\mbox{ and }\;
\mu = \mu_{ac}+\mu_{pp}.
\end{equation}
There exist at most countably many $m_i,\;M_j>0$,
and pairwise distinct $x_i\in [0,1)$ and
pairwise distinct $t_j\in (0,T]$ such that
\begin{equation}
\rho_{0,pp}= \sum\limits_{i} m_i \delta_{x_i},   \;\mbox{ and }\;
\mu_{pp}= \sum\limits_{j} M_j \delta_{t_j}.
\end{equation}
Since the measures $\rho_0$ and $\mu$ are bounded,
there exist $N_1,\,N_2\in \mathbb{N}$ such that
\begin{equation}
\label{largemass}
\sum\limits_{i> N_1} m_i < \frac{v_{\rm min}}{4L} \;\mbox{ and }\;
\sum\limits_{j> N_2} M_j < \frac{v_{\rm min}}{4L}.
\end{equation}
Without loss of generality, after possible renumbering,
we may assume that $x_{i+1}<x_i$ for all $i\leq N_1$ and
$t_j<t_{j+1}$ for all $j\leq N_2$ (the natural orderings in which
the corresponding point masses will exit the system, if they do).
Henceforth we refer to the masses  $m_i,\,i \leq N_1$ and
$M_j,\,j \leq N_2$ as {\em large masses}.\\

Due to absolute continuity, there exists $t_{00}\in (0,1)$
such that for every interval
$I\subseteq[0,1)$ of length less than $t_{00}$,
and every interval $J\subseteq [0,T]$ of length less than
$t_{00}/v_{\min}$
\begin{equation}
\label{tau0def}
\rho_{0,ac}(I) < \frac{v_{\rm min}}{4L}
\;\mbox{ and }\;
\mu_{ac}(J) < \frac{v_{\rm min}}{4L}.
\end{equation}
Let $\Omega$ be the set of strictly increasing functions
on $[0,t_{00}]$ that are Lipschitz continuous
with Lipschitz constant bounded above by $1$,
and whose inverses are Lipschitz continuous
with Lipschitz constant no larger than $v_{\rm \min}^{-1}$, that is,
\begin{equation}
\label{Omegadef}
\Omega=\{\eta:[0,t_{00}]\to [0,1] \colon \eta(0)=0,
v_{\rm \min}\leq \frac{\eta(s)-\eta(t)}{s-t}\leq 1 \text{ for all } 0\leq s<t\leq t_{00} \}.
\end{equation}
Every $\eta \in \Omega$ is absolutely continuous,
and differentiable almost everywhere.
The following properties of the set $\Omega$ are easy to prove \cite{XGthesis}.
\begin{lemma}
The set $\Omega$ defined as in \eqref{Omegadef} is closed under
taking the maximum or minimum of two curves.
That is, for every two functions $\eta_1$, $\eta_2$ in $\Omega$,
$\hat{\eta}(t) = \max\{\eta_1(t), \eta_2(t)\}$
and $\check{\eta}(t)=\min\{\eta_1(t), \eta_2(t)\}$ are in $\Omega$.
\end{lemma}
\begin{lemma}
For every fixed
$\rho_0 \in \mathcal{M}^{+}([0,1))$ and  $\mu \in \mathcal{M}^{+}((0,T])$,
the metric space $(\Omega, \|\cdot\|_{\infty})$ with $\Omega$
defined in  \eqref{Omegadef} is complete as a subspace of $C^0([0,t_{00}])$
with the supremum norm defined by
$\|\eta\|_{\infty}:=\sup_{t\in [0,t_{00}]} |\eta(t)|$.
\end{lemma}
\begin{lemma}
For fixed
$\rho_0 \in \mathcal{M}^{+}([0,1))$ and  $\mu \in \mathcal{M}^{+}((0,T])$
and $v_{\rm min}$, $0<t_{00}\leq 1$, and $\Omega$ as above,
define the map $F: \Omega \to C([0, t_{00}])$ by
\begin{equation}\label{p1eqn3}
F(\eta)(t)=\int_{0}^{t} \alpha \left(\rho_0([0, 1-\eta(s)))+\mu((0,s])\right)ds.
\end{equation}
Then for every $\eta\in \Omega$, $F(\eta)\in \Omega$.
\end{lemma}
\begin{proof}
Let $\rho_0 \in \mathcal{M}^{+}([0,1))$ and  $\mu \in \mathcal{M}^{+}((0,T])$
be arbitrary but fixed,
and let $v_{\rm min}$, $0<t_{00}\leq 1$, $\Omega$, and $F$ be as above.
Clearly, for every $\eta\in \Omega$, $F(\eta)(0)=0$.
Note for every $\bar{s}\in [0,T]$,
\begin{equation}
v_{\rm min}\leq\alpha(\rho_0([0,1))+\mu([0,T]))\leq
\alpha \left(\rho_0([0, 1-\eta(\bar{s})))+\mu((0,\bar{s}])\right) \leq 1,
\end{equation}
and thus, for all $s\neq t \in [0,t_{00}]$
\begin{equation}
v_{\rm min} \leq \frac{F(\eta)(s)-F(\eta)(t)}{s-t} \leq 1.
\end{equation}
\end{proof}
Next we show that
for any fixed measures $\rho_0\in \mathcal{M}^{+}([0,1))$
and $\mu \in \mathcal{M}^{+}((0,t_{00}])$,
whose singular continuous parts are zero,
there exists a time $\tau\in (0, t_{00}]$
and a unique Lipschitz continuous function $\xi\colon [0,\tau]\mapsto [0,1)$
such that for every $t\in [0,\tau]$,
\begin{equation}
\label{eq2contract}
\xi (t)=\int_0^t\alpha(\mu ((0,s])+\rho ([0,1-\xi (t))))\,ds.
\end{equation}
The proof, and this equation only involve a fixed measure $\rho_0$,
no mention of a curve of measures $t\mapsto \rho_t$.
In general, $\xi$ is constructed as the restriction of a curve
$\tilde\xi\in \Omega$ to a shorter time interval
$[0,\tau]\subseteq [0,t_{00}]$, and thus it inherits
the bi-Lipschitzness properties from $\tilde\xi\in \Omega$.
\begin{theorem}
\label{disthm1}
For every $\mu \in \mathcal{M}^{+}((0,T])$
and $\rho_0 \in \mathcal{M}^{+}([0,1))$
there exists a unique characteristic curve
$\xi: [0,\tau] \subseteq [0, t_{00}] \to [0,1)$
that satisfies the integral equation~\eqref{eq2contract}.
\end{theorem}
The general strategy of the proof is to modify and apply
the contraction mapping theorem, similar to~\cite{CKW2010}.
However, a naive argument breaks down over time intervals in
which large point masses exit from the system.
Thus we carefully demonstrate that the usual map is a contraction
over intervals during which no large point masses exit from the
system, and then restart the argument after a large mass has left
the system.

A technical issue is that it is not a priori known when the large
masses exit from the system at $x=1$.
This can be overcome by a nice little trick:
Replace the initial datum $\rho_0$ by a modified $\tilde\rho_0$
for which contractions can be established over a larger time
interval, and whose characteristic curves agree with those for
the original datum $\rho_0$ until the exactly computable time
when the first large mass would have exited.
Since there is only a finite number of large masses, one can
guarantee that there is a positive lower bound for the lengths
of the time intervals on which no large mass exits from the system.
Such a lower bound can be easily calculated in terms of lower
and upper bounds of the velocity $\alpha$, and
$\min \{ x_{i-1}-x_i\colon  i\leq N_1\}$ and
$\min \{ t_j-t_{j-1}\colon  i\leq N_2\}$
where we conveniently added $x_0=1$ and $t_0=0$
to the sets of $x_i$ and $t_j$ defined below in~\eqref{largemass}.
This uniform lower bound is essential in the next subsection
to guarantee a solution over
the whole interval $[0,T]$ by using only finitely many restarts.

\begin{proof}
Let $\rho_0 \in \mathcal{M}^{+}([0,1))$ and  $\mu \in \mathcal{M}^{+}((0,T])$
be arbitrary but fixed,
and let $v_{\rm min}$, $0<t_{00}\leq 1$, $\Omega$, and $F$ as
in~\eqref{vmindef},~\eqref{tau0def},~\eqref{Omegadef}, and~\eqref{p1eqn3}.
Introduce a modification $\tilde{\rho}_0$
of the initial condition $\rho_0$ such that no large masses will
leave the system in the time interval $[0,1)$. Define the new
initial condition which agrees mostly with $\rho_0$, except that
all $N_1$ large masses have been moved to $x=0$
\begin{equation}
\tilde\rho_{0,pp} = \left(\sum\limits_{i=1}^{N_1}m_i\right) \delta_{0}
                    + \sum\limits_{i>N_1} m_i\delta_{x_i}
\;\mbox{ and }\;
\tilde\rho_0 = \rho_{0,ac}+\tilde\rho_{0,pp}.
\end{equation}
Since the velocity $\alpha(W(t))$ only depends on the total load at time $t$,
the characteristic curves $\xi$ and $\tilde{\xi}$ corresponding to
initial conditions $\rho_0$ and $\tilde\rho_0$
coincide over a small time interval until $[0,\tilde\tau_1]$ defined
by $\tilde{\xi}(\tilde\tau_1)=1-x_1$ at which time the firts large mass $m_1$
would leave the original system at $x=1$.

For this initial condition $\tilde\rho_0$ (and influx $\mu$)
define $v_{\rm min}$, $t_{00}$, $\Omega$, and $F$ as
in~\eqref{vmindef},~\eqref{tau0def},~\eqref{Omegadef}, and~\eqref{p1eqn3}.
We demonstrate existence and uniqueness of a corresponding
characteristic curve $\tilde\xi$ over the time interval $[0,t_{00}]$.
For arbitrary but fixed  $\eta_1, \eta_2 \in \Omega$,
we will show that
\begin{equation}
\|F(\eta_1)-F(\eta_2)\|_{\infty} \leq \frac{1}{2} \|\eta_1-\eta_2\|_{\infty}.
\end{equation}
Since the velocity $\alpha$ is a Lipschitz continuous function
with Lipschitz constant $L$, for every fixed $t\in [0,t_{00}]$,
we have
\begin{eqnarray}
\nonumber\lefteqn{|F(\eta_1)(t) -F(\eta_2)(t)| = }
\\ \nonumber
&= &\left|\int_{0}^{t} \alpha \left(\tilde\rho_0([0, 1-\eta_1(s)))+\mu((0,s])\right)\,ds- \int_{0}^{t} \alpha \left(\tilde\rho_0([0, 1-\eta_2(s)))+\mu((0,s])\right)\,ds\right|
\\ \nonumber
&\leq &  \int_{0}^{t} \left| \rule{0mm}{3.5mm} \alpha \left(\tilde\rho_0([0, 1-\eta_1(s)))+\mu((0,s])\right)-\alpha \left(\tilde\rho_0([0, 1-\eta_2(s)))+\mu((0,s])\right)\right|\,ds
\\
&\leq & L \int_{0}^{t} \left| \rule{0mm}{3.5mm} \left(\tilde\rho_0([0, 1-\eta_1(s)))+\mu((0,s])\right)- \left(\tilde\rho_0([0, 1-\eta_2(s)))+\mu((0,s])\right)\right|\,ds
\\ \nonumber
&=& L \int_{0}^{t} \left|  \rule{0mm}{3.5mm} \tilde\rho_0([0, 1-\eta_1(s)))- \tilde\rho_0([0, 1-\eta_2(s)))\right|\,ds.
\end{eqnarray}
Choosing $\hat{\eta}(t)=\max\{\eta_1(t), \eta_2(t)\}$ and $\check{\eta}(t)=\min\{\eta_1(t), \eta_2(t)\}$,
rewrite the last expression as a double integral
\begin{eqnarray}
\lefteqn{
L \int_{0}^{t} \left|\tilde\rho_0([0, 1-\eta_1(s)))- \tilde\rho_0([0, 1-\eta_2(s)))\right|\,ds=
}\\
\nonumber
& = & L \int_{0}^{t} \left|\int_{[1-\hat{\eta}(s), 1-\check{\eta}(s))}1 \,d\tilde\rho_0(x_0)\right|\,ds
= L \int_0^t \int_{[1-\hat{\eta}(s), 1-\check{\eta}(s))}1 \,d\tilde\rho_0(x_0)\,ds.
\end{eqnarray}
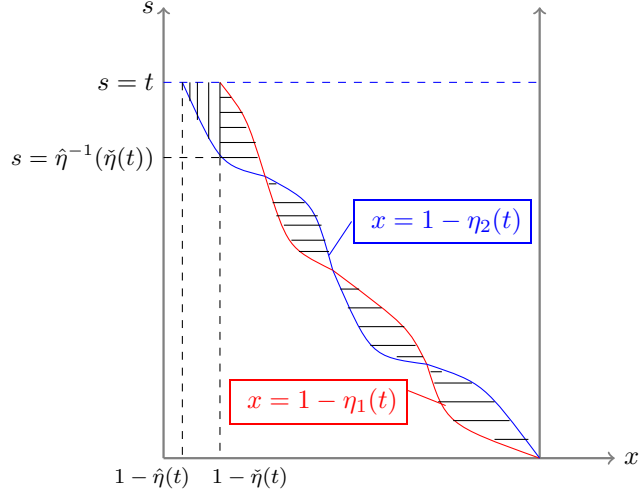
\begin{figure}[htb]
\label{xover}
\centering
\begin{tikzpicture}[scale=5.0]
\draw[help lines, line width=0.30mm, ->] (0,0)--(1.2,0) coordinate (xaxis);
\draw[help lines, line width=0.30mm, ->] (0,0)--(0,1.2) coordinate (taxis);
\draw[help lines, line width=0.30mm, ->] (1.0,0)--(1.0,1.2);
\draw[blue,dashed] (0,1.0)--(1.0,1.0);
\draw[red] (1.0, 0) .. controls (0.75,0.09) .. (0.70, 0.25);
\draw[blue] (1.0, 0) .. controls (0.85,0.20) .. (0.70, 0.25);
\draw[red] (0.70, 0.25) .. controls (0.65,0.35) .. (0.45, 0.50);
\draw[blue] (0.70, 0.25) .. controls (0.55,0.27) .. (0.45, 0.50);
\draw[blue] (0.45, 0.50) .. controls (0.40,0.68) .. (0.27, 0.75);
\draw[red] (0.45, 0.50) .. controls (0.33,0.56) .. (0.27, 0.75);
\draw[blue] (0.27, 0.75) .. controls (0.15,0.78) .. (0.05, 1.00);
\draw[red] (0.27, 0.75) .. controls (0.22,0.91) .. (0.15, 1.00);
\draw(0.15,1.00) -- (0.15, 0.80);
\draw(0.12,1.00) -- (0.12, 0.85);
\draw(0.09,1.00) -- (0.09, 0.90);
\draw(0.07,1.00) -- (0.07, 0.95);

\draw(0.15,0.96) -- (0.18, 0.96);
\draw(0.15,0.92) -- (0.21, 0.92);
\draw(0.15,0.88) -- (0.22, 0.88);
\draw(0.15,0.80) -- (0.25, 0.80);
\draw(0.15,0.84) -- (0.245, 0.84);

\draw(0.28,0.73) -- (0.30, 0.73);
\draw(0.29,0.68) -- (0.375, 0.68);
\draw(0.30,0.645) -- (0.41, 0.645);
\draw(0.36,0.55) -- (0.44, 0.55);
\draw(0.32,0.62) -- (0.42, 0.62);
\draw(0.34,0.58) -- (0.43, 0.58);

\draw(0.47,0.45) -- (0.52, 0.45);
\draw(0.49,0.40) -- (0.58, 0.40);
\draw(0.52,0.35) -- (0.64, 0.35);
\draw(0.55,0.30) -- (0.67, 0.30);
\draw(0.62,0.27) -- (0.69, 0.27);

\draw(0.71,0.23) -- (0.74, 0.23);
\draw(0.73,0.15) -- (0.88, 0.15);
\draw(0.77,0.10) -- (0.92, 0.10);
\draw(0.88,0.05) -- (0.97, 0.05);
\draw(0.72,0.20) -- (0.82, 0.20);

\draw[dashed] (0,0.80) -- (0.15, 0.80);
\node[left] at (0,0.80) {\small{$s = \hat{\eta}^{-1}(\check{\eta}(t))$}};
\node[below] at (-0.03,0.0) {\footnotesize{$1-\hat{\eta}(t)$}};
\node[below] at (0.23,0.0) {\footnotesize{$1-\check{\eta}(t)$}};
%
%

\draw[dashed](0.15,0.80) -- (0.15, 0);
\draw[dashed] (0.05, 1.00) -- (0.05, 0);
\node[left]at (taxis) {$s$};
\node[right] at (xaxis) {$x$};
\node[left] at (0,1.0) {$s=t$};
\node[right, blue] at (0.48,0.63) {\makebox{\color{blue}
\scalebox{1.0}{\framebox{
   \bf $x=1-\eta_2(t)$}}}};
\draw[blue] (0.50,0.63) -- (0.44,0.54);
\node[right, blue] at (0.15,0.15) {\makebox{\color{red}
\scalebox{1.0}{\framebox{
   \bf $x=1-\eta_1(t)$}}}};
    \draw[red] (0.65,0.18) -- (0.75,0.14);
\end{tikzpicture}
\caption{Change the Order of Integration}
\end{figure}

Since the regions are bounded by bi-Lipschitz curves, we may
change the order of integration as illustrated in figure \ref{xover}
(compare also figure 1 in \cite{CKW2010})

\begin{align}
\nonumber
L =& \int_0^t \int_{[1-\hat{\eta}(s), 1-\check{\eta}(s))}1 \,d\tilde\rho_0(x_0)\,ds
\\ \nonumber
=& L \left(\int_{[1-\hat{\eta}(t), 1-\check{\eta}(t)]}\int_{[\hat{\eta}^{-1}(1-x),t]}1\,dt \,d\tilde\rho_0(x_0)\right.
+\left. \int_{[1-\check{\eta}(t),1)}\int_{[\check{\eta}^{-1}(1-x),\hat{\eta}^{-1}(1-x))} 1 \,dt \,d\tilde\rho_0(x_0)\right)
\\ \nonumber
\leq& L \left(\int_{[1-\hat{\eta}(t), 1\!-\!\check{\eta}(t)]} (t\!-\!\hat{\eta}^{-1}(1\!-\!x))\,d\tilde\rho_0(x_0)\right.
\left. \int_{[1-\check{\eta}(t),1)} \left(\check{\eta}^{-1}(1\!-\!x)\!-\!\hat{\eta}^{-1}(1\!-\!x)\right)\,d\tilde\rho_0(x_0) \right)
\\
\leq&  L \left(\int_{[1-\hat{\eta}(t), 1\!-\!\check{\eta}(t)]} (\check{\eta}^{-1}(\check{\eta}(t)\!-\!\hat{\eta}^{-1}(\check{\eta}(t)))\,d\tilde\rho_0(x_0)\right.
\\ \nonumber
&+\left. \int_{[1\!-\!\check{\eta}(t),1)} \left(\check{\eta}^{-1}(1\!-\!x)\!-\!\hat{\eta}^{-1}(1-x)\right)\,d\tilde\rho_0(x_0) \right)
\\ \nonumber
\leq & L\left(\tilde\rho_0([1-\hat{\eta}(t), 1))\right) \sup_{0\leq y \leq \check{\eta}(t)}(\check{\eta}^{-1}(y)-\hat{\eta}^{-1}(y)).
\end{align}
To find an upper bound for the last term, use the Lipschitzness
of the curves and their inverses, and simple geometric arguments
relating the vertical offsets of the curves to their horizontal offsets.
By the definition of $\hat{\eta}$, $\check{\eta}$,
for every $y\in [0, \check{\eta}(t)]$,
we have
(compare equation $(23)$ in \cite{CKW2010})
\begin{align}
\nonumber
0 &\leq  \check{\eta}^{-1}(y)-\hat{\eta}^{-1}(y)
\\ \nonumber
&=  \left(\check{\eta}^{-1}(y) - \frac{ \hat{\eta}^{-1}(y)+\check{\eta}^{-1}(y)}{2}\right)
+\left(\frac{ \hat{\eta}^{-1}(y)+\check{\eta}^{-1}(y)}{2}- \hat{\eta}^{-1}(y)\right)
\\
&\leq  \frac{1}{v_{\rm min}} \left(y-\check{\eta} \left(\frac{ \hat{\eta}^{-1}(y)+\check{\eta}^{-1}(y)}{2}\right)\right) + \frac{1}{v_{\rm min}} \left(\hat{\eta}\left(\frac{ \hat{\eta}^{-1}(y)+\check{\eta}^{-1}(y)}{2}\right)-y\right)
\\ \nonumber
&= \frac{1}{v_{\rm min}} \left(\hat{\eta}\left(\frac{ \hat{\eta}^{-1}(y)+\check{\eta}^{-1}(y)}{2}\right)-\check{\eta} \left(\frac{ \hat{\eta}^{-1}(y)+\check{\eta}^{-1}(y)}{2}\right) \right)
\\ \nonumber
&\leq  \frac{1}{v_{\rm min}} \|\eta_1-\eta_2 \|_{\infty}.
\end{align}
Hence,
\begin{equation}
|F(\eta_1)(t)-F(\eta_2)(t)| \leq  \frac{L}{v_{\rm min}} \left(\tilde\rho_0([1-\hat{\eta}(t), 1))\right) \|\eta_1-\eta_2 \|_{\infty}.
\end{equation}
By the choice~\eqref{tau0def} of $t_{00}$,
for the absolutely continuous part $\tilde\rho_{0,ac}=\rho_{0,ac}$,
and for every $t \in [0,t_{00})$,
\begin{equation}
\tilde\rho_{0,ac}([1-\hat{\eta}(t), 1))\leq
\tilde\rho_{0,ac}([1-\hat{\eta}(t_{00}), 1))
\leq\frac{v_{\rm min}}{4L}.
\end{equation}
Note that due to their relocation and $t_{00}<1$, none of the large point masses
in $\tilde\rho_{0,pp}$ have exited during the interval $[0,t_{00}]$.
Formally, since for every $t \in [0,t_{00})$, $0 < 1-\hat{\eta}(t)$,
we conclude that
$\tilde\rho_{0,pp}([1-\hat{\eta}(t), 1))<\frac{v_{\rm min}}{4L}$.
Combining these, we obtain for $t\in [0,t_{00})$
\begin{equation}
\tilde\rho_0([1-\hat{\eta}(t),1))=\tilde\rho_{0,ac}([1-\hat{\eta}(t), 1))
+\tilde\rho_{0,pp}([1-\hat{\eta}(t), 1))<\frac{v_{\rm min}}{2L}.
\end{equation}
Hence $\frac{L}{v_{\rm min}}\tilde\rho_0([1-\hat{\eta}(t),1))<\frac{1}{2}$
showing that $F$ is a contraction on $\Omega$.
By the contraction mapping theorem, with the initial condition $\tilde\rho_0$,
there exists a unique fixed point $\tilde\xi$ in $\Omega$ such that
$\tilde\xi=F(\tilde\xi)$ over the time interval $[0,t_{00}]$.

If $\tilde\xi (t_{00}) < 1-x_1$ (the distance of the first large point mass
of the initial condition from the exit point $x=1$), then define $\xi=\tilde\xi$ on $[0,t_{00}]$
and let $\tau_1=t_{00}$.
On the other hand, if $\tilde\xi (t_{00}) \geq 1-x_1$, then
there exists a unique time $\tau_1\in(0,t_{00}]$ such that
$\tilde\xi (\tau_1) = 1-x_1$.
In this case let $\xi$ be the restriction of $\tilde\xi$
to the interval $[0,\tau_1]$.
\end{proof}

\subsection{Existence of Unique Solutions for Large Times}
\label{LargeTimes}
We start this subsection with a quick verification that
the pushforwards of the measures by the flow~$X$ (to be defined)
preserves their zero singular part.
Recall that for any Borel measurable map
$\gamma\colon S \subseteq \mathbb{R} \mapsto U\subseteq\mathbb{R}$
and any finite Borel measure $\nu \in \mathcal{M}^{+}(S)$,
the pushforward of $\nu$ by $\gamma$ is defined every Borel set
$E\subseteq U$ by
\begin{equation}
\gamma\#\nu(E) =\nu(\gamma^{-1}(E)).
\end{equation}
In the sequel we shall construct maps
\begin{equation}
X \colon \left\{(t, r)\colon 0 \leq r \leq t \leq T\right\} \times [0,\infty)
\mapsto [0, \infty); (t, r, x_0) \mapsto X(t; r, x_0)
\end{equation}
that are monotone and bi-Lipschitz in each of the first two variables,
and for fixed $(t, r) \in [0,T] \times [0, t]$ and every $x_0 \in [0, \infty)$,
$X(t; r, x_0) = X(t; r,0)+x_0$.
For such maps and for $t\in [0,T]$, the pushforward of the initial datum
$\rho_0 \in \mathcal{M}^{+}([0,1))$
by the map $X(t; 0, \cdot)\colon [0,1)\mapsto [0,\infty)$
is defined for every Borel set $E \subseteq X(t;0,\cdot)([0,1))$ by
\begin{equation}
\left(X(t; 0, \cdot)\#\rho_0\right)(E)
= \rho_0\left(X(t; 0, \cdot)^{-1}(E)\right)
=\int_{[0,1)}\chi_{E} \left(X(t; 0, x_0)\right)\,d\rho_0(x_0).
\end{equation}
Similarly, for every $t\in [0,T]$, the pushforward of the control influx
$\mu \in \mathcal{M}^{+}((0,T])$ by the map
$X(t; \cdot, 0) \colon [0, t] \mapsto [0,1)$
is defined for every Borel set  $E\subseteq [0,1)$ by
\begin{equation}
\left(X(t; \cdot,0)\#\mu\right)(E)
= \mu\left(X(t; \cdot, 0)^{-1}(E)\right)
=\int_{[0,1)}\chi_{E} \left(X(t; s, 0)\right)\,d\mu(s).
\end{equation}

\begin{lemma}\label{LemNosc}
Suppose $X(t, r, x_0)$ is as above.
Then for arbitrary Borel measures $\rho_0 \in \mathcal{M}^{+}([0,1))$
and $\mu \in \mathcal{M}^{+}((0, T])$ with zero singular continuous part,
that is, $\rho_{0, sc}=0$ and $\mu_{sc}=0$,
the singular continuous part of the pushed forward measure
\begin{equation}
X(t; 0, \cdot)\#\rho_0 + X(t; \cdot, 0)\#\mu
\end{equation}
also is zero.
\end{lemma}
\begin{proof}
%
Consider an arbitrary but fixed $t \in [0, T]$ such that $X(t;0,0)<1$.
Then for every Borel set $E \subset [0,1)$, with Lebesgue measure $\lambda(E)=0$,
let $E_1 = \left\{X(t; 0, x) \colon x \in \mbox{ supp}({\rho_{0, pp}}) \right\}$ and
$E_2 = \left\{X(t; \tau, 0) \colon \tau \in \mbox{ supp}({\mu_{pp}})\right\}$.
Let $\tilde{E}=E\setminus (E_1 \cup E_2)$. It is clear that $\lambda(\tilde{E})=0$.
Furthermore,
\begin{align}
\MoveEqLeft \left(X(t; 0, \cdot)\#\rho_0 + X(t; \cdot, 0)\#\mu\right)(\tilde{E})
            = \left(X(t; 0, \cdot)\# \rho_0\right)(\tilde{E})
            + \left(X(t; \cdot, 0)\#\mu\right)(\tilde{E})
            \nonumber\rule{22mm}{0mm}\\
\rule{8mm}{0mm}&= \rho_0\left( \left\{x\in [0,1) \colon X(t; 0, x) \in \tilde{E}\right\}\right)
   +\mu\left(\left\{\tau\in (0, t] \colon X(t; \tau, 0) \in \tilde{E})\right\}\right).
\end{align}
Let $F_1 = \left\{x\in [0,1) \colon X(t; 0, x) \in \tilde{E}\right\}
= \left\{x\in [0,1) \colon x+\xi(t) \in \tilde{E}\right\}
= \tilde{E}-\xi(t)$.
Since Lebesgue measure is translation invariant, $\lambda(F_1)=0$.
By construction of $\tilde{E}$, $\rho_0(F_1) = \rho_{0, ac}(F_1)+\rho_{0, pp}(F_1)=0+0=0$.
The set $F_2=\left\{\tau\in (0, t] \colon X(t; \tau, 0) \in \tilde{E})\right\}$
is the preimage of the set $\tilde{E}$, which has zero Lebesgue measure,
under the map $X(t; \cdot, 0): (0, t] \mapsto [0, +\infty)$.
Note that the map $X(t; \cdot, 0): (0, t] \mapsto [0, +\infty)$
is monotone and bi-Lipschitz.
Thus both $X(t; \cdot, 0)$ and its inverse are absolutely continuous.
Hence, $\lambda(F_2)=0$.
By construction of $\tilde{E}$ it follows that
$\mu(F_2)= \mu_{ac}(F_2)+\mu_{pp}(F_2)=0+0=0$.
Therefore, $\left(X(t; 0, \cdot)\#\rho_0 + X(t; \cdot, 0)\#\mu\right)(\tilde{E})=0$
which implies that the singular continuous part of $X(t; 0, \cdot)\#\rho_0 + X(t; \cdot, 0)\#\mu$
is zero.
\end{proof}
To prove the existence of unique solutions of ordinary differential
equations for large time intervals, one customarily iterates 
the fixed-point argument, with suitably modified initial data.
A key difference in our case is that the lengths of the time 
intervals on which sort-term solutions are shown to exist may 
greatly vary, depending on when large masses exit from the system.
Before we can do this, after each iteration, we construct a new
measure that serves as a parameter for the next iteration.

The key for this argument is that each iteration only involves
a single fixed measure $\rho^i$.
After a unique solution curve $\xi_i\colon [0,\tau_i]\mapsto [0,1)$
(shifted in time)
of the ordinary differential equation has been obtained,
this is used to extend the curve $t\mapsto \rho_t$ from the interval
$[0,T_{i-1}]$ to a larger interval $[0,T_i]$.

The curve $\xi$ is constructed on each interval $[T_{i-1},T_i]$
from the unique solutions of
$\dot\xi_i=\alpha(\mu^i((0,s])+\rho^i ([0,1-\xi_i(t)))$.
Afterwards,
$\rho_t$ is constructed from $\xi$ for that same time interval.
One still needs to verify that this $\xi_i$ indeed satisfies
$\dot\xi_i=\alpha(\rho_t([0,1))$ on each new interval, or that the
curve $\xi$ satisfies the related equation
$\dot{\xi}(t)=\alpha(\mu((\max\{0,\xi^{-1}(\xi(t)-1)\},t])+\rho_0([0,1-\xi(t))))$.

The total number $N$ of iterations needed to get a solution $\rho$
for all of $[0,T]$ is a priori bounded above by
${\rm ceil} (T/t_{00})+N_1+N_2$. (Only the exiting masses stemming from
$\mu$ matter, since the entering masses cancel in the contraction mapping argument.)
The maximal number  $N$ of iterations may be smaller than this bound,
e.g., if at the end $\rho_T$ still contains large point masses.
In the sequel we shall construct
\begin{itemize}
\item finite sequences $(\tau_i)_{i=0}^N$ and $(T_i)_{i=0}^{N}$
      of nonnegative numbers,
\item finite sequences of bi-Lipschitz continuous functions
      $\xi_i\colon [0,\tau_i]\mapsto [0,1)$, and
      $\xi\colon [0,T_i]\mapsto [0,\infty)$,
\item a finite sequence of maps $X\colon \{(t, r, x)\colon 0 \leq r\leq t\leq T_i, x \in [0,\infty)\} \mapsto [0,\infty)$,
\item finite sequences of measures $\rho^i \in {\cal M}^+([0,1))$
      and $\mu^i\in {\cal M}^+((0,T-T_{i-1}])$, and
\item a finite sequence of curves $\rho\colon [0,T_i]\mapsto {\cal M}^+([0,1))$.
\end{itemize}
Strictly speaking, one should also index the curves $\xi$ and $\rho$,
and the maps $X$ by $i$
as they are defined on different domains. But it will be clear that
they just denote the usual extensions of each other to larger
domains. As is customary, we omit such extra indexing.
The members of these sequences and the curves will be shown
to have the following properties for every $0\leq i \leq N$
(or $1\leq i \leq N$ for $\mu^i$, $\rho^i$). Some of these
properties will be used to construct these in the sequel.
\begin{enumerate}
\renewcommand{\theenumi}{(P\arabic{enumi})}
\setlength{\itemindent}{0.3em}
\item \label{Ptaui}
      $\;0<\tau_i<1$.
\item \label{PTi}
      $T_i=\sum_{j\leq i}\tau_j$ and $T_{N}=T$.
\item \label{Pxii}
      For every $t\in [0,\tau_i]$,
\begin{equation}
\label{diffeqn}
      \xi_i(t)=\int_0^t\alpha(\mu^i((0,s])+\rho^i ([0,1-\xi_i(s))))\,ds.
\end{equation}
\item \label{Pxi}
      For every $s\in [0,\tau_i]$,
      $\xi(T_{i-1}+s)=\xi(T_{i-1})+\xi_i(s)$.
\item \label{PX}
      For all $0 \leq r\leq s \leq t\leq T_i$ and all $x\in [0,\infty)$,
      $X(t;s,X(s;r,x))=X(t;r,x)$.
\item \label{Pmui}
      The measure $\mu^i$ is the pushforward
      (by a translation) of a restriction of the original influx, defined
      for $F\in {\cal M}^+((0,T-T_i])$ by
      $\mu^i(F)=\mu( \{t\colon t-T_{i-1} \in F\})$
\item \label{Prhot}
      For every $s\in [0,\tau_i]$, $\rho_{T_{i-1}+s}$ is the sum of
      the pushforward of a restriction of the measure $\rho^i$ (by a translation),
      and by the pushforward of a restriction
      of the measure $\mu^i$  by the map~$X$, defined
      for every Borel set $E\subseteq [0,1))$ by
      \begin{eqnarray}
      \label{changename}
      \nonumber
      \rho_{T_{i-1}+s}(E)&=&\mu^i( \{r\in (0,s] \colon X(T_{i-1}+s; T_{i-1}+r,0)\in E\})\\
                      &&+\rho^i(\{x\in [0,1)\colon X(T_{i-1}+s; T_{i-1},x)\in E\}).
      \end{eqnarray}
\item \label{Prhosemi}
      For every $0\leq t\leq T_i$ and every Borel set $E\subseteq [0,1))$
      \begin{equation}
      \label{changethis}
      \rho_{t}(E)=\mu( \{r\in (0,t] \colon X(t; r,0)\in E\})
                +\rho_0(\{x\in [0,1)\colon X(t;0,x)\in E\}).
      \end{equation}
\item \label{Pdeltax}
      For every $t\in [0,T_i]$ and for every interval $I\subseteq[0,1)$,
      if the length of $I$ is less than $t_{00}$ then
      $\;\rho_{t,ac}(I) < \frac{v_{\rm min}}{4L}$.
\item \label{Prhoi}
      The measure $\rho^i=\rho_{T_{i-1}}$ is used as the new initial
      condition for $1\leq i\leq N$.
\item \label{Prhosc}
      The singular continuous part of the measures $\mu^i$, $\rho^i$, and $\rho_t$ are zero
      (see lemma \ref{LemNosc}).
\item \label{Ptau}
      For almost every $t\in [0,T_i]$, $\dot{\xi}(t)=\alpha(\rho_t([0,1)))$.
\item \label{Pglob}
      For almost every $t\in [0,T_i]$,
      \begin{equation}
      \label{changealso}
      \dot{\xi}(t)=\alpha(\mu((\max\{0,\xi^{-1}(\xi(t)-1)\},t])+\rho_0([0,1-\xi(t)))).
      \end{equation}
      Note that if $\xi(t)\geq 1$ then
      $\rho_t([0,1-\xi(t)))=\rho_t(\emptyset)=0$.
      \vspace{1.5mm}
\end{enumerate}

For $i=0$ set $\tau_0=T_0=0$, take the trivial curves $\xi_0(0)=\xi(0)=0$ and
the identity $X(0; 0,x)=x$ for all $x\in [0,1)$.
For $i=1$ use the original measures as data $\rho^1=\rho_0$ and $\mu^1=\mu$.
Now suppose $0<i\leq N$ is arbitrary but fixed and for all $0\leq j<i$
all the above have been constructed, and have been shown to have the asserted
properties \ref{Ptaui}--\ref{Pglob}.
First define the new data $\rho^i=\rho_{T_{i-1}}$ and $\mu^i$ as in
\ref{Prhoi} and \ref{Pmui}.
Both have zero singular continuous part and their combined
total mass is less or equal to the combined mass of the
original measures $\rho_0$ and $\mu$.
In particular, the estimate \ref{Pdeltax} for the absolutely
continuous part of $\rho_{T_{i-1}}$ still holds.
Moreover, the combined number of {\em large point masses} of
$\rho^i$ and $\mu^i$ cannot exceed the combined number
of {\em large point masses} of the original measures
$\rho_0$ and $\mu$.
Thus using the same set $\Omega$ (with same $v_{\min}$ and
same uniform initial choice for $t_{00}$, theorem \ref{disthm1}
yields the existence of a $\tau_i>0$ and a unique curve
$\xi_i \colon [0,\tau_i]\mapsto [0,1)$
that satisfies \ref{Pxii}.

Now use the formula in \ref{Pxi} to extend the curve $\xi$ from
the interval $[0,T_{i-1}]$ to $[0,T_{i-1}+\tau_i]=[0,T_i]$.
Next extend the map $X$ from
$\{(r,t)\colon 0 \leq r\leq t\leq T_{i-1}\}\times [0,\infty)$
to $\{(r,t)\colon 0 \leq r\leq t\leq T_i\}\times [0,\infty)$
by first setting for all $0\leq r \leq T_{i-1}\leq t \leq T_i$
and every $x\in [0,\infty)$,
$X(t; r,x)=X(T_{i-1}; r,x)+\xi_i(t-T_{i-1})$, and then, in a second
step, for all $0 \leq T_{i-1}\leq r \leq t \leq T_i$
and every $x\in [0,\infty)$,
$X(t; r,x)=x+\xi_i(t-T_{i-1})-\xi_i(r-T_{i-1})$.
Using the property \ref{Pxi},  the above may alternatively be
written in terms of $\xi$.
Indeed, for $0 \leq r\leq T_{i-1} \leq t\leq T_i$ and $x\in [0,\infty)$
\begin{align}
X(t; r,x)&=X(T_{i-1}; r,x)+\xi_i(t-T_{i-1}) \nonumber\\
&=x+(\xi(T_{i-1})-\xi(r))+(\xi(t)-\xi(T_{i-1}))
=x+\xi(t)-\xi(r).
\end{align}
Similarly, for $0\leq T_{i-1} \leq r\leq t\leq T_i$ and $x\in [0,\infty)$
\begin{align}
X(t; r,x)&=x+\xi_i(t-T_{i-1})-\xi_i(r-T_{i-1}) \nonumber\\
&=x+(\xi(t)-\xi(T_{i-1}))-(\xi(r)-\xi(T_{i-1}))
=x+\xi(t)-\xi(r).
\end{align}
The semigroup property of the map $X$ on the larger domain follows immediately.
It is simply a consequence of the additivity of integrals
over disjoint intervals in \eqref{eq2contract}.
Let $0 \leq r\leq s \leq t\leq T_i$ and $x\in [0,1)$ be arbitrary but fixed.
Then
\begin{equation}
\label{Xsemigroup-proof}
X(t;s,X(s;r,x))=(x+\xi(s)-\xi(r))+\xi(t)-\xi(s)=x+\xi(t)-\xi(r)=X(t;r,x).
\end{equation}
Since the solution curves $\xi_i$ (and their inverses) all satisfy the
same Lipschitz bounds specified in the same set $\Omega$ (except for possible
different final times), the curve $\xi$ satisfies the same conditions. Thus
the map $X$ is Lipschitz and therefore absolutely continuous in each of
its variables.
Hence the pushforwards by scalar functions obtained from $X$
(by holding two arguments fixed) of measures are well defined.
Moreover, measures with zero singular continuous part are mapped to
measures with zero singular continuous part (lemma \ref{LemNosc}),
and absolutely continuous measures
and pure point measures mapped to measures of the same kind.

Use \eqref{changename} in item \ref{Prhot} to extend the
curve $\rho\colon t\mapsto \rho_t$ from the interval $[0,T_{i-1}]$
to the interval $[0,T_i]$.
Note that, by hypothesis, \eqref{changethis} already holds for
every $0\leq t\leq T_{i-1}$ and for every Borel set $E\subseteq [0,1)$.
In particular,
$\rho_{T_{i-1}}(E)= \mu( \{r\in (0,T_{i-1}] \colon X(T_{i-1}; r,0)\in E\})
                   +\rho(\{x\in [0,1)\colon X(T_{i-1}; 0,x)\in E\})$.
Now let $t\in [T_{i-1},T_i]$ and $E\subseteq [0,1))$ be
arbitrary but fixed Borel set.
Using the definitions of $\rho_t$ for $t$ in the new interval,
the definitions of $\rho^i$ and $\mu^i$, and the induction hypothesis, calculate:
\begin{align}
\rho_{t}(E)&=\mu^i( \{s\in (0,t-T_{i-1}] \colon X(t; T_{i-1}+s,0)\in E\})
                      +\rho^i(\{x\in [0,1)\colon X(t; T_{i-1},x)\in E\})\nonumber\\
                   =&\mu( \{r \in (T_{i-1},t] \colon X(t; r,0)\in E\})
                      +\rho_{T_{i-1}}(\{x\in [0,1)\colon X(t; T_{i-1},x)\in E\})\nonumber\\
                   =&\mu( \{r \in (T_{i-1},t] \colon X(t; r,0)\in E\})
                      +\mu( \{r \in (0,T_{i-1}] \colon X(t; T_{i-1},X(T_{i-1}; r,0)\in E\})\\
                   &+\rho_0(\{x\in [0,1)\colon X(t; T_{i-1},X(T_{i-1};0,x)\in E\})\nonumber\\
                   =&\mu( \{r \in (0,t] \colon X(t; r,0)\in E\})
                      +\rho_0(\{x\in [0,1)\colon X(t; 0,x)\in E\}).\nonumber
\end{align}
Note that there is no need to consider special cases, e.g., whether
any of $\,\xi(T_{i-1})\leq \xi(t)\leq \xi(T_i)$ is less or larger or equal to $1$.
If $\xi(t)<1$ then
$\{r \in (0,t] \colon X(t; r,0)\in [0,1)\}=(0,t]$
and $\{x\in [0,1)\colon X(t; 0,x)\in [0,1)\}$ is nonempty.
If $\xi(t)\geq 1$ then
$\{r \in (0,t] \colon X(t; r,0)\in [0,1)\}=(t-\xi^{-1}(\xi(t)-1),t]$
and $\{x\in [0,1)\colon X(t; 0,x)\in [0,1)\}$ is empty.
Using slightly different notation, taking $E=[0,1)$,
it is an immediate corollary that for all $t\in [0,T_i]$,
$\rho_{t}([0,1))=\mu(\max\{0,\xi^{-1}(\xi(t)-1)\},t)+\rho_{0}([0,1-\xi(t)))$.

Since the map $r\mapsto X(T_{i-1}+s; r,0)$ reduces distances, intuitively
$\Delta x=v\Delta t <\Delta t$ with $v\leq 1$,
the pushforward by this map of $\mu^i$
restricted to $(0,s]$ {\em concentrates} the absolutely continuous part
of $\mu$ when becoming part of $\rho_t$.
Now suppose $I\subseteq [0,1)$ is on interval of length at most $t_{00}$.
Then $J=X(T_{i-1}+s; \cdot ,0)^{-1}(I)$ is an interval of length at most
$t_{00}/v_{\min}$. Therefore $\mu_i(J)<\frac{v_{\min}}{4L}$, and thus
$\rho_{T_{i-1}}(I)<\frac{v_{\min}}{4L}$.

The next to last item is to verify that for almost every $t\in [0,T_i]$
this curve satisfies $\dot{\xi}(t)=\rho_t([0,1))$.
By hypothesis, this equation holds for almost every $t\in [0,T_{i-1}]$.
By construction, see item \ref{Pxi},
for every $s\in [0,\tau_i]$, $\xi(T_{i-1}+s)=\xi(T_{i-1})+\xi_i(s)$.
Denoting differentiation by $s$ again by a dot, using \eqref{diffeqn}
in \ref{Pxii},
it follows that at every $s\in [0,\tau_i]$ at which the integrand
of \eqref{diffeqn} is continuous
\begin{equation}
\dot\xi(T_{i-1}+s) = \dot\xi_{i}(s) =\alpha(\mu^i((0,s])+\rho^i ([0,1-\xi_i(s)))),
\end{equation}
and
\begin{align}
\rho_{T_{i-1}+s}([0,1))
=& \mu^{i}\left(\left\{ r\in (0,s] \colon X(T_{i-1}+s; T_{i-1}+r, 0) \in [0,1)\right\}\right)
\nonumber \\
&+ \rho^{i}\left(\left\{x\in [0,1) \colon X(T_{i-1}+s; T_{i-1}, x) \in [0,1)\right\}\right)
\nonumber \\
=& \mu^{i}\left(\left\{r\in (0,s] \colon \xi(T_{i-1}+s)-\xi(T_{i-1}+r)\in [0,1)\right\}\right)
\nonumber \\
&+\rho^{i}\left(\left\{x\in [0,1) \colon \xi(T_{i-1}+s)-\xi(T_{i-1})+x \in [0,1)\right\}\right)
\nonumber \\
=&\mu^{i}\left(\left\{r\in [0,s] \colon \xi(T_{i-1})+\xi_i(s)-\xi(T_{i-1})-\xi_i(t)\in [0,1)\right\}\right)
\\\nonumber
&+\rho^{i} \left(\left\{x\in [0,1) \colon \xi_i(s)+x \in [0,1)\right\}\right)
\\ \nonumber
=& \mu^{i}\left(\left\{r\in (0,s] \colon \xi_{i}(s)-\xi_i(r)\in [0,1)\right\}\right)
   +\rho^{i}([0, 1-\xi_i(s))
\\ \nonumber
=& \mu^{i}((0,s])+\rho^{i}([0, 1-\xi_i(s)).
\end{align}
Thus,
\begin{equation}
\label{xidot-xiidot}
\dot{\xi}(T_{i-1}+s)=\alpha\left(\rho_{T_{i-1}+s}([0,1))\right).
\end{equation}

This iterative procedure may be continued until $T_i=T$,
naturally working with $t_{00}$ replaced by $T-T_i$ if the
latter is smaller. Since there still may be several {\em large point masses}
exiting the system in these last intervals, there may be
several such $i$ such that $T-T_{i-1}<t_{00}$.

\subsection{The Semiflow and Lagrangian Solutions}
\label{sect31}
We fix terminology and notation for a semiflow for a vector field.
In the sequel we are only interested in the special case of a time-varying
vector field $v$ that is constant in space
(for every fixed time $t$), defined originally on $[0,T]\times [0,1)$,
but naturally extended to $[0,T]\times [0,+\infty)$.
The vector field will be integrable in time and bounded.
We make the following definition only for the case of our special
regularity hypotheses.
Thus we dispense stating the definition for more general regularity hypotheses.
\begin{definition}
\label{defflow}
Suppose $v\colon [0,T] \times [0,\infty) \mapsto [0,1)$
is integrable with respect to the first variable,
and constant with respect to the second variable.
A map $X: \{(t,r)\colon 0\leq r\leq t\;\leq T\} \times [0,\infty) \to \mathbb{R}^{+}$
is called the semiflow of the time-varying vector field $v$ if it
satisfies for all $r\in [0,T]$ and all  $x_0 \in [0,\infty)$
\begin{subequations}
\label{eq2}
\begin{align}
\label{eq2a}
\dot{X}(t; r, x_0)&=v(t,X(t; r, x_0) \;\mbox{ for almost every }\; t\in [r,T],\;\mbox{ and }
\\
\label{eq2b}
X(r; r, x_0)&= x_0
\end{align}
\end{subequations}
with $\dot{X}$ denoting the derivative of $X$ with respect to the first variable $t$.
\end{definition}

Note that in this special case the semiflow of a vector field satisfying
the stated hypotheses is clearly unique
(since $v$ is trivially Lipschitz in the space variable).

\begin{definition}
\label{LagSolDef}
Suppose $\rho_0 \in \mathcal{M}^{+}([0,1))$
and $\mu \in \mathcal{M}^{+}((0,T])$ are fixed Borel measures.
We say a function $\Phi: [0,T] \to \mathcal{M}^{+}([0,1))$
is a Lagrangian solution of the system~\eqref{eq3}
if for every $t\in [0,T]$ and every Borel set $E\subset [0,1)$
\begin{equation}
\label{geqn3}
\Phi_{t}(\rho_0,\mu)(E)=\int_{[0,1)} \chi_{E}(X(t;0,x_0))\,d\rho_0(x_{0})
+ \int_{(0,t]} \chi_{E}(X(t;s,0))\,d\mu(s).
\end{equation}
where the map
$X: \{(t,r)\colon 0\leq r\leq t\;\leq T\} \times [0,1] \to \mathbb{R}^{+}$ is the
semiflow of the vector field $v\colon(t,x)\mapsto \alpha(W(t))=\alpha(\rho_t([0,1)))$
and $\chi_{E}$ is the indicator function of set $E$.
\end{definition}
\begin{remark}
From definition \eqref{LagSolDef}, the Lagrangian solution $\Phi$
of the system \eqref{eq3} can also be interpreted as follows:
Given arbitrary but fixed Borel measures $\rho_0 \in \mathcal{M}^{+}([0,1))$
and $\mu \in \mathcal{M}^{+}((0,T])$, for every $t \in [0, T]$,
\begin{align}\Phi_t(\rho_0, \mu)= X(t; 0, \cdot)\# \rho_0 + X(t, \cdot, 0)\#\mu.
\end{align}
\end{remark}
In addition, the procedure in subsection \ref{LargeTimes} yields a semiflow
$X\colon \{(r,t)\colon 0\leq r \leq t \leq T\}\times [0,\infty) \mapsto [0,\infty)$
and a curve of positive measures
$t\mapsto \rho_t\in {\cal M}^+([0,1))$ which satisfies for almost all
$0\leq r \leq t \leq T$ and all $x\in [0,\infty)$
\begin{align}
\label{aftereqnX}
\frac{d}{dt}X(t,r,x)=&\frac{d}{dt}(x+\xi(t)-\xi(r))=\dot{\xi}(t)=\alpha(\rho_t([0,1))\nonumber\\
\nonumber
=& \alpha(\mu\left(\left\{r \in (0, t] \colon X(t; r, 0) \in [0,1)\right\}\right)
+\rho_0\left(\left\{x\in [0,1) \colon X(t; 0, x)\in [0,1)\right\}\right))\\
=&\alpha(\mu\left(\left\{r \in (0, t] \colon \left(\xi(t) -\xi(r)\right) \in [0,1)\right\}\right)
+\rho_0\left(\left\{x\in [0,1) \colon \left(x+\xi(t)\right)\in [0,1)\right\}\right))\\
=&\alpha(\mu(\max\{0,\xi^{-1}(\xi(t)-1)\},t]+\rho_0([0,1-\xi(t)))). \nonumber
\end{align}
In particular, by construction and definition \ref{LagSolDef},
the curve $t \mapsto \rho_t \in \mathcal{M}^{+}([0,1))$
is a Lagrangian solution of the system~\eqref{eq3}.
Furthermore, the existence of a unique characteristic $\xi$
implies that there is a unique semiflow $X$ that satisfies \eqref{aftereqnX}.
Thus for fixed Borel measures $\rho_0 \in \mathcal{M}^{+}([0,1))$
and $\mu\in \mathcal{M}^{+}((0,T])$, $\rho_t = \Phi_t(\rho_0, \mu)$
is the unique Lagrangian solution of the system~\eqref{eq3}.

Next we establish the semigroup property of the Lagrangian solution
of the system~\eqref{eq3} as defined in definition \ref{LagSolDef}.
For convenience, we temporarily change the notation of the Lagrangian 
solution of the system~\eqref{eq3} to emphasize the time dependence, 
and considering fixed initial data and influx. 
For arbitrary but fixed $\rho_r \in \mathcal{M}^{+}([0,1))$ and 
$\mu\in \mathcal{M}^{+}((0,T])$, consider times $0 \leq r \leq t \leq T$
Define $\mu_r \in \mathcal{M}^{+}((r, T])$ by 
$\mu_r = \mu|_{(r, T]}$ ). Then for every Borel set $E \subset [0, 1)$
write
\begin{align}
\label{pushforwardDef}
\Phi(t;r, \rho_r)(E) &= \rho_t(E)
\nonumber = \int_{[0,1)} \chi_{E}(X(t;r,x_0))\,d\rho_r(x_{0})
  + \int_{(r,t]} \chi_{E}(X(t;\tau,0))\,d\mu_r(\tau)
\\ \nonumber
&= (X(t; r, \cdot)\# \rho_r)(E)+(X(t; \cdot, 0)\# \mu_r)(E)
\end{align}
In particular, if $r=0$ then $\rho_t(E) = \Phi(t;0, \rho_0)(E)=\Phi_t(\rho_0, \mu)(E)$.
\begin{lemma}
For $0 \leq r \leq s \leq t \leq T$, the Lagrangian solution $\Phi$
of the system~\eqref{eq3} satisfies
\\
$\Phi(t; s, \Phi(s; r, \rho_r))=\Phi(t; r, \rho_r)$.
\end{lemma}
\begin{proof}
The proof uses the notation of pushforward.
By \eqref{pushforwardDef},
\begin{align}
\Phi(s, r, \rho_r) = X(s; r, \cdot)\#\rho_r + X(s; \cdot, 0)\# \mu_r.
\end{align}
Thus,
\begin{align}
\Phi(t; s, \Phi(s; r, \rho_r))
=&  X(t; s, \cdot)\#(X(s; r, \cdot)\#\rho_r + X(s; \cdot, 0)\# \mu_r) + X(t; \cdot, 0)\#\mu_s
\nonumber \\
=& X(t; s, \cdot)\#(X(s; r, \cdot)\#\rho_r)
+ X(t; s, \cdot)\#(X(s; \cdot, 0)\# \mu_r)+ X(t; \cdot, 0)\#\mu_s.
\end{align}
Let $\Phi_{1, t} = X(t; s, \cdot)\#(X(s; r, \cdot)\#\rho_r)$, then for every Borel set $E \subset [0,1)$,
\begin{subequations}
\label{rho1t}
\begin{align}
\label{rho1ta}
\Phi_{1, t}(E) &= X(t; s, \cdot)\#(X(s; r, \cdot)\#\rho_r)(E)
\\ \label{rho1tb}
&= (X(s; r, \cdot)\#\rho_r)\left(\left\{x\in [0,1) \colon X(t; s, x) \in E\right\}\right)
\\ \label{rho1tc}
&=\rho_r \left(\left\{x\in [0,1) \colon X(t; s, X(s; r, x))\in E \right\}\right)
\\ \label{rho1td}
& = \rho_r \left(\left\{x\in [0,1) \colon X(t; r, x) \in E\right\}\right)
\\ \label{rho1te}
& = (X(t; r, \cdot)\#\rho_r)(E).
\end{align}
\end{subequations}
The second to last step \eqref{rho2td} uses the semigroup property of $X$.
Therefore, $\Phi_{1, t} =X(t; r, \cdot)\#\rho_r $.
\\
Let $\Phi_{2, t}=X(t; s, \cdot)\#(X(s; \cdot, 0)\# \mu_r)+ X(t; \cdot, 0)\#\mu_s$.
Then for every Borel set $E \subset [0,1)$,
\begin{subequations}
\label{rho2t}
\begin{align}
\label{rho2ta}
\Phi_{2, t}(E)=& X(t; s, \cdot)\#(X(s; \cdot, 0)\# \mu_r)(E)+ X(t; \cdot, 0)\#\mu_s(E)
\\ \label{rho2tb}
=& (X(s; \cdot, 0)\# \mu_r) \left(\left\{x\in [0,1) \colon X(t;s, x) \in E\right\}\right)
+ X(t; \cdot, 0)\#\mu_s(E)
\\ \label{rho2tc}
=& \mu_r \left(\left\{ \tau \in (r, s] \colon X(t; s, X(s; \tau, 0 )) \in E \right\}\right)
+X(t; \cdot, 0)\#\mu_s(E)
\\ \label{rho2td}
=&\mu_r\left(\left\{\tau \in (r, s] \colon X(t;\tau, 0) \in E\right\}\right)
+ \mu_s \left(\left\{\tau \in (s, t] \colon X(t; \tau, 0)\in E\right\}\right)
\\ \label{rho2te}
=& \int_{(r, s]} \chi_E(X(t; \tau, 0))d\mu(\tau)+\int_{(s, t]} \chi_E(X(t; \tau, 0))d\mu(\tau)
\\ \label{rho2tf}
=& \int_{(r, t]} \chi_E(X(t; \tau, 0))d\mu(\tau)
\\ \label{rho2tg}
=& (X(t; \cdot, 0)\#\mu_r)(E)
\end{align}
\end{subequations}
The fourth step \eqref{rho2td} uses the semigroup property of the semiflow $X$. 
Hence, $\Phi_{2,t}=X(t; \cdot, 0)\#\mu_r$.
Therefore, $\Phi(t; s, \Phi(s; r, \rho_r))=X(t; r, \cdot)\#\rho_r+X(t; \cdot, 0)\#\mu_r = \Phi(t; r, \rho_r).$
\end{proof}

\section{Definition, Existence and Uniqueness of Weak Solutions}
\label{sect4}

Motivated by the Lagrangian point of view, we define a new notion
of weak solution of equations~\eqref{eq3}
and prove that the Lagrangian solution
is indeed a unique measure-valued weak solution.

\subsection{Definition of Weak Solutions}

This section defines a notion of weak solution of the
hyperbolic conservation law~\eqref{eq3}.
Let $\Psi$ be the set of functions $\varphi \colon [0,T] \times [0,1) \mapsto \mathbb{R}$
such that for every $t \in [0, T]$, $\varphi(t, 1)=0$, $\varphi(t, \cdot)$
is differentiable and $\frac{\partial \varphi}{\partial x}$ is continuous (jointly in $(t,x)$)
and for every $x\in [0,1)$, $\varphi(\cdot, x)$ is Lipschitz continuous. That is,
\begin{align}
\Psi =& \left\{ \varphi\colon [0,T] \times [0,1)\mapsto \mathbb{R} \mid \text{ for every } t \in [0, T],
\varphi(t, 1) =0,~ \varphi(t, \cdot) \text{ is differentiable, }\right.
\nonumber \\
& \left. \partial_x \varphi \text{ is continuous } (\text{jointly in }(t,x)), and \right.
\\ \nonumber
&\left. \text{ for every } x\in [0,1), \varphi(\cdot, x) \text{ is Lipschitz continuous }\right\}.
\end{align}

\begin{definition}
\label{defweaksolution}
A measure-valued weak solution of equation~\eqref{eq3}
with initial condition $\rho_0 \in \mathcal{M}^{+}([0,1))$
and boundary condition $\mu \in \mathcal{M}^{+}((0,T])$
is a function $\rho: [0,T] \to \mathcal{M}^{+}([0,1))$,
such that $W \colon [0,T] \mapsto \rho_t([0,1))$ is integrable and
such that for every $\tau \in [0,T]$
and for every $\varphi \in \Psi$, one has
\begin{align}
\label{geqn4}
\nonumber
0&= \int_{(0,\tau]} \int_{[0,1)}
\left(\partial_t \varphi(t,x)+\alpha(W(t))\partial_x\varphi(t,x)\right) \,d\rho_{t}(x)\,dt
+ \int_{(0,\tau]} \varphi (t,0) \,d\mu(t)\\
&\;\;\;\;-\int_{[0,1)} \varphi(\tau,x)\,d\rho_{\tau}(x)
+ \int_{[0,1)} \varphi(0,x) \,d\rho_0(x).
\end{align}
\end{definition}
\subsection{Existence of the Weak Solution}
In this subsection we show the existence
of the measure-valued weak solutions to equation~\eqref{eq3}.
We start with the following lemma about a ``weak chain rule''.

\begin{lemma}
Let $\rho_0 \in \mathcal{M}^{+}([0,1))$ and
$\mu \in \mathcal{M}^{+}((0,T])$ be arbitrary but fixed.
Let $X: \{(t,r)\colon 0\leq r\leq t\;\leq T\} \times [0,1) \to \mathbb{R}^{+}$
be the semiflow of the vector field $v\colon(t,x)\mapsto \alpha(W(t))$
with $W(t)=\Phi_t(\rho_0, \mu)([0,1))$.
Then for almost every $t \in [0, T]$ and every $x_0 \in [0,1)$,
every test function $\varphi$ in $\Psi$ satisfies
\begin{equation}
\frac{d\varphi(t, X(t; 0, x_0))}{dt}
= \partial_{t}\varphi(t, X(t; 0, x_0))+\alpha(W(t))\partial_{x}\varphi(t, X(t; 0, x_0)).
\end{equation}
\end{lemma}	
\begin{proof}
Fix a test function $\varphi \in \Psi$. We show that for almost every $t\in [0,T]$,
every $x_0 \in [0,1)$, arbitrary but fixed $\varepsilon>0$, there exists $\delta >0$,
such that, if $\left|\Delta t\right| < \delta$, then
\begin{align}
&\left| \frac{\varphi(t+\Delta t, X(t+\Delta t; 0, x_0))-\varphi(t, X(t; 0, x_0))}{\Delta t}
-\partial_{t}\varphi(t, X(t; 0, x_0))-\alpha(W(t))\partial_{x}\varphi(t, X(t; 0, x_0))\right|
< \varepsilon.
\end{align}
Since for every $x\in [0,1)$, $\varphi(\cdot, x)$ is Lipschitz continuous and thus differentiable
almost everywhere, the map $\xi \colon [0, T] \mapsto [0, \infty)$ is also differentiable almost
everywhere.
Fix arbitrary $(t, x_0) \in [0,T] \times [0,1)$ such that  both $\partial_t\varphi(t, X(t; 0, x_0))$
and $\partial_t X(t; 0, x_0) = \dot{\xi}(t)$ exist.
Let $\varepsilon>0$ be arbitrary but fixed.
Since $\partial_t\varphi(t, X(t; 0, x_0))$ exists, there exists $\delta_1 >0$
such that,  if $\left|\Delta t\right| < \delta_1$, then
\begin{equation}
\label{eqn1afterD}
\left|\frac{\varphi(t+\Delta t, X(t; 0, x_0))-\varphi(t; X(t; 0, x_0))}{\Delta t}
 - \partial_t\varphi(t, X(t; 0, x_0)) \right| < \frac{\varepsilon}{4}.
\end{equation}
Since for every $t + \Delta t \in [0, T]$,
the map $\varphi(t+\Delta t, \cdot)\colon [0,1) \mapsto \mathbb{R}$
is differentiable, there exists $\delta_2>0$,
such that if $\left|\Delta x \right| < \delta_2$, then
\begin{equation}
\label{eqn2afterD}
\left|\frac{\varphi(t+\Delta t, X(t; 0, x_0)+\Delta x)
-\varphi(t+\Delta t, X(t; 0, x_0))}{\Delta x}
- \partial_x\varphi(t+\Delta t, X(t, 0, x_0))\right| < \frac{\varepsilon}{4}.
\end{equation}
Note that if $\left|\Delta t\right|< \delta_2$,
then $\left|X(t+\Delta t; 0, x_0) - X(t; 0, x_0)\right| \leq \left|\Delta t\right|
< \delta_2$. Let $\Delta x = X(t+\Delta t; 0, x_0) - X(t; 0, x_0)$.
From \eqref{eqn2afterD}, we obtain
\begin{equation}
\label{eqn3afterD}
\left|\frac{\varphi(t+\Delta t, X(t+\Delta t; 0, x_0))-
\varphi(t+\Delta t, X(t; 0, x_0))}{X(t+\Delta t; 0, x_0) - X(t; 0, x_0)}
- \partial_x\varphi(t+\Delta t, X(t, 0, x_0))\right| < \frac{\varepsilon}{4}.
\end{equation}
In addition, since $\partial_x \varphi$ is continuous on $[0, T] \times [0, 1]$
and hence bounded, from \eqref{eqn3afterD}, there exists $U>0$ such that
\begin{equation}
\label{eqn5afterD}
\left|\frac{\varphi(t+\Delta t, X(t+\Delta t; 0, x_0))
-\varphi(t+\Delta t, X(t; 0, x_0))}{X(t+\Delta t; 0, x_0) - X(t; 0, x_0)}\right| < U.
\end{equation}
Since $\partial_t X(t; 0, x_0) = \dot{\xi}(t)$ exists at $t$,
there exists $\delta_3 >0$ such that, if $\left|\Delta t \right|<\delta_3$, then
\begin{equation}
\label{eqn6afterD}
\left|\frac{X(t+\Delta t; 0, x_0) - X(t; 0, x_0)}{\Delta t}
- \alpha(W(t))\right| < \frac{\varepsilon}{4U}.
\end{equation}
Since $\partial_x \varphi$ is continuous (jointly in $(t,x)$),
there exists $\delta_4>0$ such that, if $\left|\Delta t\right|<\delta_4$,
then
\begin{equation}
\label{eqn4afterD}
\left|\partial_x \varphi(t+\Delta t, X(t; 0, x_0))-\partial_x \varphi(t, X(t; 0, x_0))\right|
< \frac{\varepsilon}{4}.
\end{equation}
Choose $\Delta t $ such that $\left|\Delta t\right|< \min\{\delta_i \colon i=1, 2, 3, 4 \}$.
Then from \eqref{eqn1afterD}, \eqref{eqn3afterD}, \eqref{eqn5afterD}, \eqref{eqn6afterD}
and \eqref{eqn4afterD}, and $\alpha(W(t)) \in [0,1)$, we obtain
\begin{align}
\MoveEqLeft \left| \frac{\varphi(t+\Delta t, X(t+\Delta t; 0, x_0)
)-\varphi(t, X(t; 0, x_0))}{\Delta t}
-\partial_{t}\varphi(t, X(t; 0, x_0))-\alpha(W(t))\partial_{x}\varphi(t, X(t; 0, x_0))\right|
\nonumber \\
\leq &\left|\frac{\varphi(t+\Delta t, X(t; 0, x_0))-\varphi(t; X(t; 0, x_0))}{\Delta t}
 - \partial_t\varphi(t, X(t; 0, x_0)) \right|
\nonumber \\
&+\left|\left(\frac{\varphi(t+\Delta t, X(t+\Delta t; 0, x_0))
-\varphi(t+\Delta t, X(t; 0, x_0))}{X(t+\Delta t; 0, x_0)
- X(t; 0, x_0)}\right) \alpha(W(t))\right.
\nonumber\\
&- \left. \partial_x\varphi(t+\Delta t, X(t, 0, x_0)) \alpha(W(t))\right|
\\ \nonumber
&+\left|\frac{\varphi(t+\Delta t, X(t+\Delta t; 0, x_0))
-\varphi(t+\Delta t, X(t; 0, x_0))}{X(t+\Delta t; 0, x_0) - X(t; 0, x_0)}\right|
\\ \nonumber
& \quad \left|\frac{X(t+\Delta t; 0, x_0) - X(t; 0, x_0)}{\Delta t}
- \alpha(W(t))\right|
\\ \nonumber
&+\left|\left(\partial_x \varphi(t+\Delta t, X(t; 0, x_0))
-\partial_x \varphi(t, X(t; 0, x_0))\right)\alpha(W(t))\right|
<\varepsilon.
\end{align}
By the definition of differentiability, for almost all $t\in [0, T]$
\begin{align}
\frac{d\varphi(t, X(t; 0, x_0))}{dt}
= \partial_{t}\varphi(t, X(t; 0, x_0))+\alpha(W(t))\partial_{x}\varphi(t, X(t; 0, x_0)).\
\end{align}
\end{proof}		
The following theorem guarantees the existence of measure-valued weak solutions.
\begin{theorem}
Every Lagrangian solution of~\eqref{geqn3}
is a weak solution that satisfies~\eqref{geqn4}.
\end{theorem}
\begin{proof}
For sake of simplicity of notation, we conveniently extend the functions 
$\varphi$, $\partial_{t}\varphi$ and $\partial_{x}\varphi$ to $[0,\infty)$ 
with value zeros for $x\geq 1$. (Of course these extended functions generally 
are not differentiable at $x=1$);.
Suppose $\Phi$ is a Lagrangian solution that satisfies~\eqref{geqn3},
evaluate the right hand side of equation~\eqref{geqn4} at $\rho = \Phi$,
\begin{align}
\nonumber
\lefteqn{\int_{(0,\tau]} \int_{[0,1)}\left(\partial_t \varphi(t,x)
+\alpha(W(t))\partial_x\varphi(t,x)\right) d\Phi_t(\rho_0,\mu)(x)\,dt}
\\
&= \int_{(0,\tau]} \int_{[0,1)}\left(\partial_t \varphi(t,X(t;0,x_0))
+\alpha(W(t))\partial_x\varphi(t,X(t;0,x_0))\right) \,d\rho_0(x_0)\,dt
\\ \nonumber
\mbox{}&\;\;\;\;+\int_{(0,\tau]} \int_{(0,t]}\left(\partial_t \varphi(t,X(t;s,0))
+\alpha(W(t))\partial_x\varphi(t,X(t;s,0))\right) \,d\mu(s)\,dt
\end{align}
Note that
\begin{subequations}
\label{Lwk1}
\begin{align}
\label{Lwk1a}
\lefteqn{\int_{(0,\tau]} \int_{[0,1)}\left(\partial_t \varphi(t,X(t;0,x_0))
+\alpha(W(t))\partial_x\varphi(t,X(t;0,x_0))\right) \,d\rho_0(x_0)\,dt}
\\ \label{Lwk1b}
&= \int_{(0,\tau]} \int_{[0,1)} \frac{d\varphi(t, X(t;0,x_0))}{dt} \,d\rho_0(x_0)\,dt
\\ \label{Lwk1c}
&= \int_{[0,1)} \int_{(0,\tau]} \frac{d\varphi(t, X(t;0,x_0))}{dt} \,dt\,d\rho_0(x_0)
\\ \label{Lwk1d}
&= \int_{[0,1)} \left(\varphi(\tau, X(\tau; 0, x_0))-\varphi(0, X(0; 0, x_0))\right)\,d\rho_0(x_0)
\\ \label{Lwk1e}
&=\int_{[0,1)} \varphi(\tau,X(\tau;0,x_0))\,d\rho_0(x_0)
- \int_{[0,1)} \varphi(0, X(0; 0, x_0)) \,d\rho_0(x_0)
\\ \label{Lwk1f}
&= \int_{[0,1)} \varphi(\tau,X(\tau;0,x_0))\,d\rho_0(x_0)- \int_{[0,1)} \varphi(0, x) \,d\rho_0(x).
\end{align}
\end{subequations}
On the other hand,
\begin{subequations}
\label{Lwk2}
\begin{align}
\label{Lwk2a}
\lefteqn{\int_{(0,\tau]} \int_{(0,t]}\left(\partial_t \varphi(t,X(t;s,0))
+\alpha(W(t))\partial_x\varphi(t,X(t;s,0))\right) \,d\mu(s)\,dt}
\\ \label{Lwk2b}
&= \int_{(0,\tau]} \int_{(0,t]} \frac{d\varphi(t, X(t;s,0))}{dt} \,d\mu(s)\,dt
\\ \label{Lwk2c}
&= \int_{(0,\tau]} \int_{(s,\tau]} \frac{d\varphi(t, X(t;s,0))}{dt} \,dt \,d\mu(s)
\\ \label{Lwk2d}
&= \int_{(0,\tau]} \left(\varphi(\tau, X(\tau; s,0))-\varphi(s, X(s;s,0))\right)\,d\mu(s)
\\ \label{Lwk2e}
&= \int_{(0,\tau]} \varphi(\tau,X(\tau;s,0))\,d\mu(s)- \int_{(0,\tau]} \varphi(t, 0) \,d\mu(t).
\end{align}
\end{subequations}
In addition,
\begin{equation}
\int_{[0,1)} \varphi(\tau,x)\,d\Phi_{\tau}(\rho_0,\mu)(x)
= \int_{[0,1)} \varphi(\tau,X(\tau;0,x_0))\,d\rho_0(x_0)+\int_{[0,1)} \varphi(\tau,X(\tau;s,0))\,d\mu(s).
\end{equation}
Thus, every Lagrangian solution~\eqref{geqn3} is a weak solution.
\end{proof}
\subsection{Uniqueness of the weak solution}
Next we show that every measure-valued weak solution is also a Lagrangian solution
to the hyperbolic conservation law \eqref{eq3}.
From theorem \eqref{disthm1} we obtain the uniqueness of the measure-valued weak solution.
\begin{theorem}
\label{thmunique}
Foe every initial condition $\rho_0 \in \mathcal{M}^{+}([0,1))$
and for every  boundary condition $\mu \in \mathcal{M}^{+}((0, T])$,
the measure-valued weak solution \eqref{geqn4} to \eqref{eq3} is unique.
\end{theorem}
\begin{proof}
We first establish the uniqueness of the measure-valued weak solution $\hat{\rho}$
over a small time interval $[0,\tau_0]$ first, with $\tau_0$ defined as in \eqref{tau0def}.
By the definition of the measure-valued weak solution,
$\hat{W}\colon t\mapsto \hat{\rho}_t([0,1))$ is integrable, 
and for arbitrary but fixed $\tau \in (0,\tau_0]$, 
for every $\varphi \in \Psi$ and $t\in [0,\tau]$
\begin{align}\label{uniqeqn2}
0&=\int_{(0,\tau]}\int_{[0,1)} \left(\varphi_t(t,x)
+\alpha(\hat{W}(t))\varphi_x(t,x)\right)\,d\hat{\rho}_t(x)dt
+ \int_{(0,\tau]}\varphi(t,0)\,d\mu(t)
\nonumber\\
& \;\;\;-\int_{[0,1)}\varphi(\tau,x)\,d\hat{\rho}_{\tau}(x)
+ \int_{[0,1)} \varphi(0,x)\,d\rho_0(x),
\end{align}
Consider a $C^1$ function $\varphi_0\in C_0^1(0,1)$ with compact support in $(0,1)$,
and let $\hat{\xi}(t) = \int_0^t \alpha(\hat{W}(s))\,ds$ for $t\in [0,\tau]$.
Choose the test function
\begin{align}
    \varphi(t,x)& =
    \begin{cases}
       \varphi_0(\hat{\xi}(\tau)-\hat{\xi}(t)+x),
       & \text{ if } 0\leq x \leq \hat{\xi}(t)-\hat{\xi}(\tau)+1, 0\leq t \leq \tau,\\
       0,& \text{ if } 0\leq \hat{\xi}(t)-\hat{\xi}(\tau)+1 \leq x \leq 1,0\leq t \leq \tau.
    \end{cases}
\end{align}
Note that every test function $\varphi \in\Psi$ satisfies the Cauchy problem
\begin{align}
    \begin{cases}
      \partial_t \varphi + \alpha(\hat{W}(t))\partial_{x}\varphi =0,
      & \text{ if } 0\leq t \leq \tau, 0\leq x \leq 1,\\
       \varphi(\tau,x)=\varphi_0(x),& \text{ if } 0\leq x \leq 1.\\
    \end{cases}
\end{align}
From \eqref{uniqeqn2}, we obtain
\begin{align}
 \int_{[0,1)}\varphi_0(x)\,d\hat{\rho}_{\tau}(x)=
 \int_{(0,\tau]}\varphi_0(\hat{\xi}(\tau)-\hat{\xi}(t))\,d\mu(t)
 +\int_{[0,1-\hat{\xi}(\tau))} \varphi_0(\hat{\xi}(\tau)+x)\,d\rho_0(x).
\end{align}
Since $\varphi_0 \in C_0^1(0,1)$ and $\tau \in [0,\tau_0]$ were arbitrary,
for every Borel set $E \subset [0,1)$, and every $t\in [0,\tau_0]$,
\begin{align}
\hat{\rho}_{t}(E)
&= \int_{(0,t]}\chi_{E}(\hat{\xi}(t)-\hat{\xi}(s))\,d\mu(s)
+\int_{[0,1)} \chi_{E}(\hat{\xi}(t)+x)\,d\rho_0(x).
\end{align}
Therefore,
\begin{align}
\hat{W}(t) =\hat{\rho}_{t}([0,1))
&=\int_{(0,t]}\chi_{[0,1)}(\hat{\xi}(t)-\hat{\xi}(s))\,d\mu(s)
+\int_{[0,1)} \chi_{[0,1)}(\hat{\xi}(t)+x)\,d\rho_0(x)
\nonumber \\
&= \mu((0, t])+\rho_0([0, 1-\hat{\xi}(t)).
\end{align}
Furthermore,
\begin{align}
\hat{\xi}(t) &= \int_0^t \alpha(\hat{W}(s))\,ds
=\int_0^t \alpha(\mu((0, s])+\rho_0([0, 1-\hat{\xi}(s)))\,ds=F(\hat{\xi})(t).
\end{align}
It is easy to check that $\hat{\xi}\in \Omega$.
Since $\xi$ is the unique fixed point of the function $F\colon \Omega\mapsto \Omega$,
$\hat{\xi}=\xi$. Thus $\hat{\rho_t} = \rho_t$ over the time interval $[0,\tau_0]$
which implies the uniqueness of the weak solution of \eqref{eq3}
over the time interval $[0,\tau_0]$.
Similar to the proof of theorem \eqref{disthm1}, one obtains the uniqueness
of the weak solution of \eqref{eq3} defined on all of $[0,T]$ in a finite number of iterations.
\end{proof}
\begin{remark}
We chose the set of test functions $\Psi$ in definition~\ref{defweaksolution}
to match the Lipschitz properties of  $\hat{\xi}$  and $\varphi_0 \in C_0^1(0, 1)$.
\end{remark}
\subsection{Regularity of the weak solution} \label{sect43}
For the special case of finite signed Borel measures on the
interval $[0,1)$ briefly recall the definition of the flat norm.
First denote by $\mathcal{F}$ the set of nonnegative Lipschitz
continuous functions with Lipschitz constant $1$ and that are 
bounded above by $1$
\begin{equation}
\label{Lip1}
\mathcal{F}=\{  f \colon [0,1) \mapsto [0,1] \colon \text{ for all } x,y \in [0,1),
 |f(x)-f(y)| \leq |x-y| \}.
\end{equation}
On the space~$\mathcal{M}([0,1))$ of signed measures, define for $\nu \in  \mathcal{M}([0,1))$
the flat norm
\begin{equation}\label{flat}
\|\nu\|_\flat = \sup_{f\in \mathcal{F}} \left|\int_{[0,1)} f\,d\nu\right|.
\end{equation}
For applications and a discussion of properties the flat norm
in general settings ~\cite{Thieme2018} is an excellent reference.
For a careful discussion comparing and mild solutions and weak solution 
see, e.g., \cite{Evers16}.

The following simple examples shows that in general
the solution~$t\mapsto \rho_t$  of~\eqref{eq3} need not be continuous
under the flat norm $\|\,\cdot\,\|_\flat$, and indeed
with respect to any norm on $\mathcal{M}^{+}([0,1))$.
Example \ref{delta0} illustrates that problems with jumps in the outflux
can be handled by working with a modification of the flat norm defined
as defined in \eqref{normeqn1}.
Example \ref{delta1} demonstrates shows that problems with jumps in the
influx cannot be handled by any norm, but we can still have continuity with
respect to a similar seminorm defined in \eqref{semieqn1}.
\begin{example}
\label{delta0}
Consider the case of $\alpha (W)=\frac{1}{1+W}$, the trivial influx $\mu = 0$ and
the initial datum $\rho_0=\delta_0$ consisting of a single unit mass $\delta_0$
at $x=0$. Then for every $t<2$, the solution $\rho_t=\delta_{t/2}$
consists of a single mass at $x=\frac{1}{2}t$, whereas for all
$t\geq 2$, the solution $\rho_t=0$ is the trivial measure.
Using the function $f\equiv 1$ it is easily seen that for every~$t<2$
we have $\|\rho_t-\rho_2\|_\flat = \|\delta _{t/2}\|_\flat = 1$,
and hence $t\mapsto \rho_t$ is not continuous at $t=2$ with respect to the flat norm.
\end{example}

\begin{example}
\label{delta1}
Consider the case of $\alpha (W)=\frac{1}{1+W}$,
the trivial initial datum $\rho_0=0$ and influx $\mu = \delta_1$,
consisting of a single unit mass $\delta_1$ at $t=1$.
Then for every $1\leq t < 3$ the solution $\rho_t=\delta_{(t-1)/2}$
consists of a single mass at $x=\frac{1}{2}(t-1)$, whereas for all
other $t$ the solution $\rho_t=0$ is the trivial measure.
Let $\|\cdot \|$ be any norm on $\mathcal{M}^{+}([0,1))$,
and set $\varepsilon =\|\delta_0\|$.
Then every open neighborhood $V$ of $t=1$ contains a time $t\in V$
such that $\|\rho_t-\rho_1\| = \|\delta_1\| \nless \varepsilon.$
\end{example}

In order to discount the importance of the exact time at which
point masses enter the system or exit from it, we introduce
a modification of the flat norm.
First, define the piecewise linear {\em hat}-function
$h:\mathbb{R} \mapsto \mathbb{R}$ by
\begin{equation}
\label{hatfcn}
h(x)= \left\{
\begin{array}{cc}
\frac{1}{2}-|\frac{1}{2}-x| & \text{ if } x \in [0,1)\\
0 & \text{otherwise}
\end{array}
\right.
\end{equation}
Define the  map $\phi: \mathcal{M}([0,1)) \mapsto [0,\infty)$ by
\begin{equation}
\label{semieqn1}
\phi(\nu) = \sup_{f\in \mathcal{F}} \left|\int_{[0,1)} fh\,d\nu\right|.
\end{equation}
\textbf{Note: }
If $f$ and $h$ are both nonnegative Lipschitz continuous functions
with Lipschitz constant~$1$  and both are bounded above by~$1$ then
their product $fh$ is Lipschitz with Lipschitz constant~$2$ and is
also bounded by~$1$. It is immediate that:

\begin{lemma}
The map $\phi$ defines a seminorm on the space $\mathcal{M}([0,1))$.
\end{lemma}

\begin{example}
Continuing the example~\ref{delta0}, it is easy to calculate
that for all $t<2$ one has
$\phi (\rho_t-\rho_2)=\frac{1}{2}-\left|\frac{1}{2}-\frac{t}{2}\right|$
while for all $t\geq 2$ one has $\phi (\rho_t-\rho_2)=0$,
and for these special data the solution $t\mapsto \rho_t$ is a
continuous curve in $\mathcal{M}([0,1))$ when endowed with
the seminorm $\phi$.
\end{example}

\begin{theorem}
For every $\mu \in \mathcal{M}^{+}((0,T])$ and $\rho_0 \in \mathcal{M}^{+}([0,1))$,
the unique solution $\rho \colon [0,T] \mapsto \mathcal{M}^{+}([0,1)))$
of system~\eqref{eq3}, and thus also of \eqref{geqn4},
is continuous under the seminorm $\phi$.
\end{theorem}
\begin{proof}
Let $T>0$, $\mu \in \mathcal{M}^{+}((0,T])$, $\rho_0 \in \mathcal{M}^{+}([0,1))$
be arbitrary but fixed
and let $\rho \colon [0,T] \mapsto \mathcal{M}^{+}([0,1)))$
be the unique solution of system~\eqref{eq3}, and thus also of \eqref{geqn4},
and let $X$ be the associated semiflow.
Without loss of generality consider times $0 \leq t_2<t_1 \leq T=1$.
(For times larger than $1$, the continuity follows from the
semiflow property of $t\mapsto \rho_t$, via composition of
continuous functions.)

By the choice of the time $T\leq 1$, there exists locations
$0\leq x_1< x_2<1$ such that $X(t_1;0,x_1)=X(t_2;0,x_2)=1$.
We show that for every $\varepsilon >0$, if $t_1-t_2$
is sufficiently small, then $\phi(\rho_{t_1}-\rho_{t_2} )<\varepsilon$.
In particular, for any arbitrary fixed $f\in \mathcal{F}$
we find an upper bound for
$\left|\int_{[0,1)} fh\,d(\rho_{t_1}-\rho_{t_2})\right|$.
For those {\em parts} that are in the factory at both
times $t_2$ and $t_1$, a simple Lipschitz estimate will
do the job. However, for parts that entered, or exited from
the factory between these times, we use that for all
$x\in [0,1)$, $\;h(x)\leq x$ and $h(x)\leq 1-x$.
The first step uses that $\rho_t$ is constructed from
the pushforwards of the data $\rho_0$ and $\mu$.
\begin{subequations}
\begin{align}
\MoveEqLeft \left|\int_{[0,1)} fh\,d(\rho_{t_1}-\rho_{t_2})\right|
     = \left|\int_{[0,1)} f(x)h(x)\,d\rho_{t_1}(x)
            -\int_{[0,1)} f(x)h(x)\,d\rho_{t_2}\right|
\\ 
= & \left| \int_{[0,x_1)} (fh)(X(t_1;0,x_0)) \,d\rho_0(x_0)
      - \int_{[0,x_2)} (fh)(X(t_2;0,x_0)) \,d\rho_0(x_0)\right.
\\ \nonumber
& \left. +\int_{(0,t_1]} (fh)(X(t_1;s,0)) \,d\mu(s)
          -\int_{(0,t_2]} (fh)(X(t_2;s,0)) \,d\mu(s) \right|
\\ 
\leq &  \int_{[0,x_1)} \left| (fh)(X(t_1;0,x_0))-(fh)(X(t_2;0,x_0))\right| \,d\rho_0(x_0)
+ \int_{[x_1,x_2)} (fh)(X(t_2;0,x_0)) \,d\rho_0(x_0)
\\ \nonumber
& +\int_{(0,t_2]} \left|(fh)(X(t_1;s,0))-(fh)(X(t_2;s,0))\right| \,d\mu(s)
+\int_{(t_2,t_1]} (fh)(X(t_2;s,0)) \,d\mu(s)
\\ 
\label{twointeg}
\leq & \;2\int_{[0,x_1)} \left| X(t_1;0,x_0)-X(t_2;0,x_0)\right| \,d\rho_0(x_0)
       + \int_{[x_1,x_2)} h(X(t_1;0,x_0)) \,d\rho_0(x_0)
\\ \nonumber
& +2\int_{(0,t_2]} \left|X(t_1;s,0)-X(t_2;s,0)\right| \,d\mu(s)
    +\int_{(t_2,t_1]} h(X(t_2;s,0)) \,d\mu(s).
\end{align}
\end{subequations}
In the last step, the first and third integral in \eqref{twointeg}
use the Lipschitz constant $2$ for $fh$, whereas the
other two use that $f$ is bounded above by 1.
For the first and third integral in \eqref{twointeg} use that the semiflow $X$ is
Lipschitz continuous (for fixed second and third variables)
with Lipschitz constant $1$, and hence the integrals are
bounded above $(t_1-t_2)\cdot \rho([0,x))$ and
$(t_1-t_2)\cdot \mu((0,T])$, respectively.
For the second integral in \eqref{twointeg} note that for every $x\in [x_1,x_2)$,
$\;X(t_2;0,x)\geq 1-(t_1-t_2)$ and hence the integral is bounded
above by $(t_1-t_2)\cdot \rho([0,x))$.
For the fourth integral note that for every $s\in [t_2,t_1]$,
$X(t_1;s,0)< t_1-t_2$n and hence the integral is bounded above
by $(t_1-t_2)\cdot \mu((0,T])$.

Thus given any $\varepsilon >0$,
choose $\delta =\varepsilon/(2\rho([0,1))+2\mu((0,T])$.
For all $0\leq t_2\leq t_1<1$, if $t_1-t_2<\delta$,
then for every $f\in \mathcal{F}$,
$\left|\int_{[0,1)} fh\,d(\rho_{t_1}-\rho_{t_2})\right|<\varepsilon$
and hence $\phi(\rho_{t_1}-\rho_{t_2})\leq \varepsilon$.
\end{proof}
This establishes continuity of the solution $t\mapsto \rho_t$
using the seminorm $\phi$.
The ultimately desirable joint continuity of the semiflow with
respect to time, the influx, and the initial conditions
appears elusive.
However, we present in theorem~\ref{almost} that a useful property
of the semiflow, though it is not continuity with respect to initial
conditions.
This uses a slightly different seminorm as illustrated
in the following example.

\begin{example}
\label{g-no-good}
Consider the case of $\alpha (W)=\frac{1}{1+W}$, trivial influx $\mu = 0$ and
trivial initial datum $\tilde\rho_0=0$, $0<T\leq 1$,
$\varepsilon =\frac{T}{2}$ and $0<\delta<1$ arbitrary but fixed.
Set $x_0=\frac{\delta}{2}$
and let $\rho_0\in \mathcal{M}^{+}([0,1))$ be the measure consisting
of the unit point mass at $x_0$.
Then $\phi(\rho_0-\tilde\rho_0)=x_0<\delta$, yet the respective
solutions at time $T$ are $\tilde\rho_T=0$ and
$\rho_T$ consisting of a unit point mass at $(x_0+\frac{1}{2}T)$
and hence
$\phi(\rho_T-\tilde\rho_T)=x_0+\frac{1}{2}T \nless\frac{1}{2}T=\varepsilon.$
\end{example}
As established in the example above, the seminorm $\phi$
used (and needed) to establish continuity of the solution $\rho_t$
with respect to time, will not provide continuity with
respect to initial conditions using the seminorm~$\phi$.
However, using a similar seminorm that only discounts
variations close to the exit point $x=1$ appears better suited.
In analogy with \eqref{hatfcn}
define  $g\colon {\mathbb R}\mapsto \mathbb R$ by $g(x)=1-x$
and correspondingly to \eqref{semieqn1} define a modification $\psi$
of the flat norm on $\mathcal{M}([0,1))$ by
\begin{equation}\label{normeqn1}
\psi(\nu) = \sup_{f\in \mathcal{F}} \left|\int_{[0,1)} fg\,d\nu\right|.
\end{equation}

Recall the following lemma (Proposition $2.2$ from \cite{Thieme2018}).
\begin{lemma}
\label{afterDlem1}
The indicator function of every closed (open) set in $[0,1)$
is the pointwise limit of a decreasing  (respectively increasing)
sequence of bounded Lipschitz continuous functions $(f_n)$,
where each $f_n$ has Lipschitz constant $n$ and takes values between $0$ and $1$.
\end{lemma}
\begin{theorem}
The modification $\psi$  of the flat norm defined in \eqref{normeqn1} is a norm.
\end{theorem}
\begin{proof}
Clearly, \eqref{normeqn1} defines a seminorm.
It remains to be verified that for every
$\nu \in \mathcal{M}^{+}([0,1))$, $\psi(\nu)$ implies $\nu =0$.

Let $\nu \in \mathcal{M}^{+}([0,1))$ be arbitrary but fixed such that $\psi(\nu)=0$.
Then for every Borel set $T\subseteq [0,1)$
and for every function $f$ that is Lipschitz continuous with Lipschitz constant $1$
and with values between $0$ and $1$, $\int_{T} fg \, d\nu=0$,
where $T \subseteq [0,1)$ is a Borel set.
In addition, for arbitrary but fixed $\varepsilon>0$,
there exists $\delta \in (0, 1)$ such that $\nu((1-\delta, 1))<\varepsilon$.
Since $\nu$ is regular, for every Borel set $A \subseteq [0,1)$,
\begin{align}
\nu(A) = \inf \left\{\nu(G)\colon A \subseteq G \subseteq [0,1), G \text{ is an open set }\right\}.
\end{align}
Now fix Borel set $A \subseteq [0,1)$ and let $G \supseteq A$ be an open set in $[0,1)$. Then
\begin{align}
\nu(G) = \nu(G\cap [0, 1-\delta]) +\nu (G\cap (1-\delta, 1))<\nu(G \cap [0, 1-\delta]) + \varepsilon.
\end{align}
By lemma \ref{afterDlem1}, there exists an increasing sequence
of bounded Lipschitz continuous functions $(f_n)$ such that $f_n \to \chi_{G}$ pointwise.
Let $f_n^{*} = \frac{1}{g} f_n$.
Since the function $\frac{1}{g}$ is continuously differentiable
and bounded above by $\frac{1}{\delta}$
over the interval $[0, 1-\delta]$, $f_n^{*}$ is bounded above
by $\frac{1}{\delta}$ and Lipschitz continuous with Lipschitz constant
$n\cdot \frac{1}{\delta}+1\cdot \frac{1}{\delta^2}$
over the interval $[0, 1-\delta]$.
Note that $\frac{1}{\frac{1}{\delta} \left(n+\frac{1}{\delta}\right)} f_n^{*}$ is Lipschitz continuous with Lipschitz constant $1$ and with values between $0$ and $1$ on $[0, 1-\delta]$.
Thus,
\begin{align}
& \int_{[0,1-\delta]} f_n^{*} g \, d\nu = 0.
\end{align}
By the Monotone Convergence Theorem,
\begin{align}
\int_{[0, 1-\delta]} \chi_G \,d\nu
&= \lim\limits_{ n\to \infty} \int_{[0, 1-\delta]} f_n \, d\nu
= \lim\limits_{n\to \infty} \int_{[0, 1-\delta]} f_n^{*}:wq
g \, d\nu.
\end{align}
Thus,
\begin{align}
[\nu(G \cap [0, 1-\delta]) = \int_{[0,1-\delta]} \chi_{G}\,d\nu =0.
\end{align}
Therefore, $\nu(G)< \varepsilon$ which implies that $\nu(A)=0$ and thus $\nu=0$.
\end{proof}

The following example illustrates how this norm $\psi$ avoids
the problems of the seminorm $\phi$ with regards to continuity
with respect to initial conditions.

\begin{example}
\label{h-maybe-good}
Consider the case of $\alpha (W)=\frac{1}{1+W}$, trivial influx $\mu = 0$ and
initial data $\rho_{10}$ and $\rho_{20}$ consisting
of point masses of sizes $M\geq m\geq 1$ located at $0\leq a\leq b<\frac{1}{2}$,
respectively. Then
\begin{align}
\psi( \rho_{20}-\rho_{10}) & =M(1-a)-m(1-b)+m(1-b)(b-a)
\\ & =(M-m)(1-a)+m(b-a)(2-b).
\end{align}
Suppose that $\delta > \psi( \rho_{20}-\rho_{10})$.
Then, in particular, $M-m <2\delta$ and $b-a<\delta$.
At any small time $0\leq t\leq 1$ (before either mass
exists the system), the measures $\rho_{1t}$ and $\rho_{2t}$
are point masses of sizes $M$ and $m$ at the locations
$(a+\frac{t}{1+M})\leq (b+\frac{t}{1+m})$, respectively.
It is easily seen that
\begin{equation}
\psi( \rho_{2t}-\rho_{1t})=
M(1-(a+\frac{t}{1+M}))-m(1-(b+\frac{t}{1+m}))(1-((b+\frac{t}{1+m})-(a+\frac{t}{1+M}))).
\end{equation}
Evaluating this at $m=1$, $M=m+x$, $b=a+y$ gives a simple rational
expression in $x,y,t,\delta$ whose numerator vanishes at $x=y=0$
(and denominator bounded away from zero). In particular if $2|x|,|y|<\delta$
then
\begin{equation}
\psi( \rho_{2t}-\rho_{1t})\leq
\frac{2\delta(2\delta^2+(4a+2t-2)\delta+\frac{1}{2}t^2+(a+2)t+8a-12)}{8+4\delta}.
\end{equation}
Thus it is clear that for every $t>0$ and every $\varepsilon >0$
it is possible to choose $\delta>0$ such that if the initial data
$M,m=1,a,b$
as above satisfy  $\psi( \rho_{20}-\rho_{10})<\delta$, then
$\psi( \rho_{2t}-\rho_{1t})<\varepsilon$.
\end{example}
This example shows that replacing the seminorm $\phi$ by the
norm $\psi$ on $\mathcal{M}([0,1))$ provides some hope for
continuity with respect to initial conditions.
This norm preserves the features of $\phi$ by discounting
the importance of the specific exit times of large masses, but it
avoids the trouble presented in example~\ref{h-maybe-good}.
We have not been able to show that, in general,
the semiflow $(t,\rho_0)\mapsto \rho_t$
is continuous with respect to the initial datum $\rho_0$
and the norm~$\psi$. However,
we have the following result which is weaker than continuity.

\begin{theorem}
\label{almost}
For every $\mu \in \mathcal{M}^{+}((0,T])$ and every $\rho_0 \in \mathcal{M}^{+}([0,1))$,
the unique weak solution $\rho \colon [0,T] \mapsto \mathcal{M}^{+}([0,1))$
of system~\eqref{eq3} satisfies:
for every initial condition $\tilde{\rho_0}\in \mathcal{M}^{+}([0,1))$
and every $\varepsilon>0$, there exist $\delta >0$ and $\tau>0$ such that
if $\phi(\tilde{\rho_0}-\rho_0)<\delta$ and $t<\tau$,
then  $\;\phi(\tilde{\rho_t}-\rho_t)<\varepsilon$.
\end{theorem}
\begin{proof}
Consider the control influx $\mu$ and two initial conditions
$\rho_0^{1}$, $\rho_0^{2} \in \mathcal{M}^{+}([0,1))$.
For $k=1, 2$, denote by $\rho_{t}^k$, $W_k$ and $X_k$ the weak measure-valued solution,
the total load and the semiflow corresponding to the initial condition~$\rho_0^{k}$ respectively.
For convenience, for $t \in [0,T]$, let $\xi_k (t) = X_k(t; 0, 0)$
be the characteristic curve as in section \ref{sect31}. Thus
\begin{align}
    \begin{cases}
       \xi_k'(t) &= \alpha(W_k(t)) \text{ for a. e. } t \in [0,T],\\
       \xi_k(0) &= 0.
    \end{cases}
\end{align}
Since the velocity $\alpha_k$ is positive and bounded above by $1$,
for every $t\in [0,T]$, we have
\begin{align}
\left|\xi_1(t)- \xi_2(t)\right| \leq \int_0^t \left|\alpha(W_1(s)) - \alpha(W_2(s))\right|\,ds \leq t.
\end{align}
Furthermore, for every $x_0 \in [0,1)$,
\begin{align}
\left|X_1(t; 0, x_0) - X_2(t; 0, x_0) \right| = \left|\xi_1(t) - \xi_2(t) \right| \leq t,
\end{align}
and for every $s\in [0,t]$,
\begin{subequations}
\begin{align}
\left|X_1(t; s, 0) - X_2(t; s, 0) \right| & = \left|\xi_1(t)-\xi_1(s) -\left(\xi_2(t)-\xi_2(s)\right) \right|
\\
& = \left|\xi_1(t)-\xi_2(t) + \xi_2(s)-\xi_1(s)\right|
\\
& \leq \left|\xi_1(t)-\xi_2(t)\right| + \left|\xi_1(s)-\xi_2(s)\right|
\\ 
& \leq t+s \leq 2t.
\end{align}
\end{subequations}
In addition, there exist $t_1, t_2 \in [0,1]$ such that $X_1\left(t_1; 0, \frac{1}{2}\right)=1$
and   $X_2\left(t_2; 0, \frac{1}{2}\right)=1$.
For an arbitrary but fixed $\varepsilon>0$, consider the time interval $[0,\tau]$ where
\begin{align}
\tau = \min \left\{1, t_1, t_2,\frac{\varepsilon}{15\left(\rho_0^1([0,1))
+\rho_0^2([0,1))\right)},\frac{\varepsilon}{10\rho_0^2([0,1))}, \frac{\varepsilon}{20\mu([0,T])}\right\}.
\end{align}
For arbitrary but fixed $t \in [0,\tau]$ there exists locations $x_0^1, x_0^2 \in [0,1)$
such that $X_1(t; 0, x_0^1)=1$ and $X_2(t; 0, x_0^2)=1$.
Without loss of generality, we assume that $x_0^1 < x_0^2$.
Note that $x_0^1, x_0^2 \in (\frac{1}{2}, 1)$.
Next we show that if $\delta =\frac{\varepsilon}{5}>0$,  then for every
$t \in [0,\tau]$, if $\phi(\rho_0^1 - \rho_0^2) < \delta$,
then $\phi(\rho_t^{1} - \rho_t^{2})<\epsilon$.
For arbitrary but fixed $f \in \mathcal{F}$, and for every $t\in [0,\tau]$, we have
\begin{subequations}
\begin{align}
\MoveEqLeft \left|\int_{[0,1)} fh\,d(\rho_{t}^{1}-
\rho_{t}^{2}) \right| = \left|\int_{[0,1)} f(x)h(x)\,d\rho_{t}^{1}(x)
-\int_{[0,1)} f(x)h(x)\,d\rho_{t}^{2}(x)\right|
\\
= & \left|\int_{[0,x_0^1)} (fh)(X_1(t; 0,x_0))\,d\rho_0^1(x_0)
-\int_{[0,x_0^2)} (fh)(X_2(t;0,x_0))\,d\rho_0^{2}(x_0)\right. 
\\ \nonumber 
& \left. +\int_{[0,t)} (fh)(X_1(t;s,0))\,d\mu(s)
- \int_{[0,t)} (fh)(X_2(t;s,0))\,d\mu(s) \right|
\\
=& \left|\int_{[0,x_0^1)} (fh)(X_1(t; 0,x_0))\,d\rho_0^1(x_0)
-\int_{[0,x_0^1)} (fh)(X_1(t; 0,x_0))\,d\rho_0^2(x_0) \right.
\\ \nonumber
& +\left. \int_{[0,x_0^1)} (fh)(X_1(t; 0,x_0))\,d\rho_0^2(x_0)
-\int_{[0,x_0^1)} (fh)(X_2(t;0,x_0))\,d\rho_0^{2}(x_0)\right.
\\  \nonumber
& - \left. \int_{[x_0^1,x_0^2)} (fh)(X_2(t; 0,x_0))\,d\rho_0^2(x_0) \right. 
\\ \nonumber
& \left. +\int_{[0,t)} (fh)(X_1(t;s,0))\,d\mu(s)
- \int_{[0,t)} (fh)(X_2(t;s,0))\,d\mu(s) \right|
\\ \label{coneqn1hh}
 \leq & \left| \int_{[0,x_0^1)} (fh)(X_1(t; 0,x_0))\,d\rho_0^1(x_0)
 -\int_{[0,x_0^1)} (fh)(X_1(t; 0,x_0))\,d\rho_0^2(x_0) \right|
\\ \label{coneqn1h}
& + \int_{[0,x_0^1)} \left| (fh)(X_1(t; 0,x_0))-(fh)(X_2(t;0,x_0))\right|\,d\rho_0^{2}(x_0)
\\ \label{coneqn1i}
& + \int_{[x_0^1,x_0^2)} (fh)(X_2(t; 0,x_0))\,d\rho_0^2(x_0)
\\ \label{coneqn1j}
& +\int_{[0,t)} \left| (fh)(X_1(t;s,0)) - (fh)(X_2(t;s,0))\right|\,d\mu(s). 
\end{align}
\end{subequations}
By the triangle inequality, we obtain,
\begin{subequations}
\begin{align}
\MoveEqLeft  \left| \int_{[0,x_0^1)} (fh)(X_1(t; 0,x_0))\,d\rho_0^1(x_0)
-\int_{[0,x_0^1)} (fh)(X_1(t; 0,x_0))\,d\rho_0^2(x_0) \right|
\nonumber 
\\
&=\left|\int_{[0, x_0^1)}(fh)(X_1(t; 0, x_0)) \, d(\rho_0^1-\rho_0^2)(x_0)\right|
\\
&=   \left|\int_{[0, x_0^1)}(fh)(x_0) \, d(\rho_0^1-\rho_0^2)(x_0)
 +\int_{[0, x_0^1)}\left((fh)(X_1(t; 0, x_0))-(fh)(x_0)\right) \, d(\rho_0^1-\rho_0^2)(x_0)\right|
\\ 
&\leq  \left|\int_{[0, x_0^1)}(fh)(x_0) \, d(\rho_0^1-\rho_0^2)(x_0)\right|
 +\left|\int_{[0, x_0^1)}\left((fh)(X_1(t; 0, x_0))-(fh)(x_0)\right) \, d(\rho_0^1-\rho_0^2)(x_0)\right|.
\end{align}
\end{subequations}
Let $I_1 = \left|\int_{[0, x_0^1)}(fh)(x_0) \, d(\rho_0^1-\rho_0^2)(x_0)\right|$. Then
\begin{align}
I_1 & \leq \left|\int_{[0,1)} (fh) (x_0)\, d(\rho_0^1 - \rho_0^2) (x_0)\right|
+ \left|\int_{[x_0^1, 1)} (fh)(x_0) \, d(\rho_0^1 - \rho_0^2)(x_0)\right|
\\ \nonumber
& \leq \left|\int_{[0,1)} (fh) (x_0)\, d(\rho_0^1 - \rho_0^2) (x_0)\right|
+ \left|\int_{[x_0^1, 1)}(fh)(x_0)\,d\rho_0^1(x_0)\right|
+\left|\int_{[x_0^1, 1)}(fh)(x_0)\,d\rho_0^2(x_0)\right|.
\end{align}
Using that the function $f$ is bounded above by $1$,
$h$ is decreasing over $(\frac{1}{2}, 1)$ and $x_0^1 \in (\frac{1}{2}, 1)$, we have,
\begin{align}
\left|\int_{[x_0^1, 1)}(fh)(x_0)\,d\rho_0^1(x_0)\right|
& \leq \left|\int_{[x_0^1, 1)}h(x_0)\,d\rho_0^1(x_0)\right|
 \leq h(x_0^1) \rho_0^1([0,1))
= (1-x_0^1)\rho_0^1([0,1))
\leq t \rho_0^1([0,1)).
\end{align}
Similarly, we obtain
\begin{align}
\left|\int_{[x_0^1, 1)}(fh)(x_0)\,d\rho_0^1(x_0)\right| \leq t \rho_0^2([0,1)).
\end{align}
Using properties of the seminorm $\phi$,
\begin{align}
I_1 \leq \phi(\rho_0^1 - \rho_0^2) + t \left(\rho_0^1([0,1))+\rho_0^2([0,1))\right).
\end{align}
Let $I_2 =\left|\int_{[0, x_0^1)}\left((fh)(X_1(t; 0, x_0))-(fh)(x_0)\right)
\, d(\rho_0^1-\rho_0^2)(x_0)\right| $.
Using that the function $fh$ is Lipschitz continuous with Lipschitz constant $2$,
it follows that
\begin{align}
I_2 \leq & \left| \int_{[0, x_0^1)}\!\left((fh)(X_1(t; 0, x_0))
\!-\!(fh)(x_0)\right) \, d(\rho_0^1)(x_0)\right|
\!+\! \left| \int_{[0, x_0^1)}\!\left((fh)(X_1(t; 0, x_0))
\!-\!(fh)(x_0)\right) \, d(\rho_0^2)(x_0)\right|
\nonumber \\
\leq & \int_{[0, x_0^1)}\left|(fh)(X_1(t; 0, x_0))-(fh)(x_0)\right| \, d(\rho_0^1)(x_0)
+ \int_{[0, x_0^1)}\left|(fh)(X_1(t; 0, x_0))-(fh)(x_0)\right| \, d(\rho_0^2)(x_0)
\nonumber  \\
\leq & 2 \left(\int_{[0, x_0^1)} \left|X_1(t; 0, x_0)-x_0\right| \, d\rho_0^1(x_0)
+ \int_{[0, x_0^1)} \left|X_1(t; 0, x_0)-x_0\right| \, d\rho_0^2(x_0)\right)
\\ \nonumber
\leq & 2t\left(\rho_0^1([0,1))+\rho_0^2([0,1))\right).
\end{align}
Therefore, by the definition of the seminorm $\phi$,
the integral \eqref{coneqn1hh} is bounded above by
$\phi(\rho_0^1 - \rho_0^2)+3t\left(\rho_0^1([0,1))+\rho_0^2([0,1))\right)$.
Using that the function $fg$ is Lipschitz continuous
with Lipschitz constant $2$, we obtain that the 
integral \eqref{coneqn1h} is bounded above by
\begin{align}
2\int_{[0,1)} \left|X_1(t; 0, x_0)-X_2(t;0,x_0)\right|\,d\rho_0^{2}(x_0) \leq 2t\rho_0^2([0,1)).
\end{align}
The integral \eqref{coneqn1j} is bounded above by
\begin{equation}
2\int_{[0,t)} \left|X_1(t; s, 0)-X_2(t;s,0)\right|\,d\mu(s) \leq 4t\mu([0,T)).
\end{equation}
For the integral in \eqref{coneqn1i},
using that the function $f$ is bounded above by $1$
\begin{equation}
\int_{[x_0^1,x_0^2)} (fh)(X_2(t; 0,x_0))\,d\rho_0^2(x_0)
\leq \int_{[x_0^1,x_0^2)} h(X_2(t; 0,x_0))\,d\rho_0^2(x_0).
\end{equation}
Since the semiflow $X_2$ is increasing with respect to the third variable,
for $x_0 \in [x_0^1, x_0^2)$,
$X_2(t; 0, x_0) \in [X_2(t; 0, x_0^1), 1)\subset (\frac{1}{2}, 1)$.
Note that the function $h$ decreases over the interval $(\frac{1}{2}, 1)$.
Thus

\begin{align}
\MoveEqLeft \int_{[x_0^1,x_0^2)} h(X_2(t; 0,x_0))\,d\rho_0^2(x_0)
\leq \int_{[x_0^1,x_0^2)} h(X_2(t; 0, x_0^1))\, d\rho_0^2(x_0) 
\nonumber \\
& =  \int_{[x_0^1,x_0^2)} \left(1 -X_2(t; 0, x_0^1)\right)\, d\rho_0^2(x_0)
\leq \int_{[x_0^1,x_0^2)} \left(1 - x_0^1\right)\, d\rho_0^2(x_0)\nonumber\leq t\rho_0^2([0,1)).
\end{align}
The last inequality above is due to the fact that the velocity $\alpha_1$ is bounded above~by~$1$.
\end{proof}

\section{Conclusion and outlook}
\label{sect5}

We substantially relaxed the regularity hypotheses under which
well-posedness is guaranteed for the model~\eqref{npeqn1}
from~\cite{AM2006} for highly re-entrant semi-conductor manufacturing
systems, a model that has spawned much follow-up research.
By closely adhering to the features of the original industrial
problem, primarily by first focusing on the Lagrangian point of view,
we established well-posedness for Borel measure-valued data.

Attending to this specific system allowed us to delineate the boundaries
of what regularity properties may be expected  and possible to prove for
more general settings.
We presented a modification of usual weak solutions that uses a modified
set of test functions whose weaker regularity properties are adapted to
this system.
Pushing the envelope, we established continuity of the semiflow
with respect to time for a semi-norm, and proved that the semiflow is
generally not continuous for any norm.
We presented further partial results for continuity with respect
to initial conditions. Joint continuity using a single
metric does not appear possible. Further continuity properties
such as input-to state, and input to output may be subject of
future work.

The generalization to  vector-valued measures (shared machines for
different products) appears to be straightforward. More interesting
are generalizations to weighted contributions of the work-in-progress
which correspond to dispatch polices like PUSH and PULL that are often
implemented on the factory floor.
The continued interest in such weighted work in progress models,
applying to also multiple products is illustrated in~\cite{keimer2018}.
It is a nice challenge to combine these with weak solutions.
Another major challenge is to present a Pontryagin-like Maximum
Principle that applies to hyperbolic conservation laws such as our model,
with either $L^1$-data or measure-valued data.

\section*{Acknowledgments}
We gratefully acknowledge the feedback,
many invaluable suggestions, corrections by 
Steve Kaliszewski, Sebastien Motsch, Hal Smith, Horst Thieme who served
as members of the first author's Ph.D. committee, and Michael Herty,
the external reviewer of the first author's thesis.

\bibliographystyle{plain}
\bibliography{referenceDec}
\end{document}